\documentclass[12pt]{amsart}
\usepackage[english]{babel}
\usepackage[utf8x]{inputenc}
\usepackage{amsfonts}
\usepackage{amsmath}
\usepackage{longtable}
\usepackage{graphicx}
\usepackage{enumerate}
\usepackage{tikz-cd}
\usepackage{bbm}
\usepackage{mathtools}
\usepackage{amsthm}
\usepackage{amssymb}
\usepackage{url,smartref}
\usepackage{mathrsfs} 
\usepackage{tikz,comment}
\usepackage{extpfeil,bm,fancyhdr,mathbbol}
\usepackage[top=1in, bottom=1.25in, left=1.25in, right=1.25in]{geometry}
\usetikzlibrary{shapes, arrows}
\tikzstyle{terminator} = [rectangle, draw, text centered, rounded corners, minimum height=2em]
\usetikzlibrary {arrows.meta,bending,positioning}

\usepackage{nameref}
\usepackage{enumitem, hyperref}

\makeatletter
\def\namedlabel#1#2{\begingroup
   \def\@currentlabel{#2}%
   \label{#1}\endgroup
}
\makeatother

\makeatletter
\newcommand{\labitem}[2]{%
\def\@itemlabel{\textbf{#1}}
\item
\def\@currentlabel{#1}\label{#2}}
\makeatother

\numberwithin{equation}{section}

\usepackage{multirow}
\usepackage{subfigure}
\usepackage{algorithm}
\usepackage{algorithmic}
\usepackage{comment,mathrsfs}

\title{On the growth of Tate--Shafarevich groups of $p$-supersingular abelian varieties of $\gl_2$-type over $\bbZ_p$-extensions of number fields}
\author{Erman I{\c S}IK and Antonio Lei}


\address{University of Ottawa, Department of Mathematics and Statistics, STEM Complex, 150 Louis-Pasteur Pvt, Ottawa, Ontario, K1N 6N5, Canada}
\email{\url{eisik@uottawa.ca},  \url{antonio.lei@uottawa.ca}}
\urladdr{\url{https://sites.google.com/view/erman-isik/}\\\url{https://antoniolei.com/}}

\newcommand{\orange}[1]{{\color{orange} \sf  #1}}

\newcommand{\mylabel}[2]{#2\def\@currentlabel{#2}\label{#1}}
\DeclareFontFamily{U}{wncy}{}
    \DeclareFontShape{U}{wncy}{m}{n}{<->wncyr10}{}
    \DeclareSymbolFont{mcy}{U}{wncy}{m}{n}
    \DeclareMathSymbol{\Sh}{\mathord}{mcy}{"58}

\newcommand{\Zp}{\bbZ_p}

\newcommand{\Qp}{\bbQ_p}

\newcommand{\image}{\mathrm{Image}}
\theoremstyle{definition}
\newtheorem{definition}{Definition}[section]
\newtheorem{lemma}[definition]{Lemma}
\newtheorem{theorem}[definition]{Theorem}
\newtheorem{prop}[definition]{Proposition}
\newtheorem{corollary}[definition]{Corollary}
\newtheorem{conj}[definition]{Hypothesis}

\newtheorem{remark}[definition]{Remark}

\newtheorem{thmx}{Theorem}

\DeclareFontFamily{U}{wncy}{}
    \DeclareFontShape{U}{wncy}{m}{n}{<->wncyr10}{}
    \DeclareSymbolFont{mcy}{U}{wncy}{m}{n}
    \DeclareMathSymbol{\Sh}{\mathord}{mcy}{"58}

\newcommand{\HIw}{H^1_{\mathrm{Iw}}}

\newcommand{\hatF}{\widehat{F}}
\newcommand{\hatA}{\widehat{A}}
\newcommand{\hatG}{\widehat{G}}

\newcommand{\cO}{\mathcal{O}}
\newcommand{\ccS}{\mathcal{S}}

\newcommand{\ccP}{\mathcal{P}}

\newcommand{\ccC}{\mathcal{C}}

\newcommand{\ccO}{\mathcal{O}}

\newcommand{\ccH}{\mathcal{H}}

\newcommand{\ccX}{\mathcal{X}}
\newcommand{\ccY}{\mathcal{Y}}

\newcommand{\uCol}{\underline\col}

\newcommand{\bbG}{{\mathbb G}}

\newcommand{\bbF}{{\mathbb F}}

\newcommand{\bbQ}{{\mathbb Q}}

\newcommand{\bbZ}{{\mathbb Z}}

\newcommand{\bfc}{{\mathbf c}}
\newcommand{\bff}{{\mathbf f}}

\newcommand{\bfz}{{\mathbf z}}

\newcommand{\bfx}{{\mathbf x}}
\newcommand{\bfy}{{\mathbf y}}

\newcommand{\bfs}{{\mathbf s}}

\newcommand{\bbd}{{\mathbb d}}

\newcommand{\scrF}{{\mathscr F}}

\newcommand{\fraf}{{\mathfrak f}}

\newcommand{\fram}{{\mathfrak m}}

\newcommand{\frap}{{\mathfrak p}}
\newcommand{\fraq}{{\mathfrak q}}

\newcommand{\mat}{ M_{d \times d}(\Lambda_{\ccO})}

\newcommand{\An}{A_{(n)}}

\newcommand{\gal}{{\rm Gal}}
\newcommand{\tr}{{\rm Tr}}

\newcommand{\Hom}{{\rm Hom}}
\newcommand{\End}{{\rm End}}
\newcommand{\Sel}{{\rm Sel}}

\newcommand{\loc}{{\rm loc}}
\newcommand{\col}{{\rm Col}}

\newcommand{\iw}{{\rm Iw}}

\newcommand{\ord}{{\rm ord}}

\newcommand{\rank}{{\rm rank}}

\newcommand{\gl}{{\rm GL}}

\newcommand{\coker}{{\rm coker}}

\newcommand{\length}{{\rm length}}

\newcommand{\corank}{{\rm corank}}

\newcommand{\Ld}{ \Lambda^{\oplus d}}
\newcommand{\Ldn}{ \Lambda_{\ccO,n}^{\oplus d}}
\newcommand{\Lvn}{ \Lambda_{v,n}^{\oplus d}}

\usepackage[OT2,T1]{fontenc}
\DeclareSymbolFont{cyrletters}{OT2}{wncyr}{m}{n}
\DeclareMathSymbol{\Sha}{\mathalpha}{cyrletters}{"58}

\definecolor{Green}{rgb}{0.0, 0.5, 0.0}

\newcommand{\ZZ}{\mathbb{Z}}
\begin{document}
\maketitle

\begin{abstract}
     We study the boundedness of the Mordell--Weil rank and the growth of the $v$-primary part of the Tate--Shafarevich group of $p$-supersingular abelian varieties of ${\rm GL}_2$-type with real multiplication over $\mathbb Z_p$-extensions of number fields, where $v$ is a prime lying above $p$.  Building on the work of Iovita--Pollack in the case of elliptic curves, under precise ramification and splitting conditions on $p$, we construct explicit systems of local points using the theory of Lubin--Tate formal groups. We then define signed Coleman maps, which in turn allow us to formulate and analyse signed Selmer groups. Assuming these Selmer groups are cotorsion, we prove that the Mordell--Weil groups are bounded over any subextensions of the $\bbZ_p$-extension and provide an asymptotic formula for the growth of the $v$-primary part of the Tate--Shafarevich groups. Our results extend those of Kobayashi, Pollack, and Sprung on $p$-supersingular elliptic curves.
\end{abstract}

\section{Introduction}

One of the central themes in Iwasawa theory is the study of the asymptotic behaviour of arithmetic invariants along towers of number fields. In his foundational work \cite{Iwa73}, Iwasawa showed that for a number field $ K $ and a $ \mathbb{Z}_p $-extension $ K_\infty $ over $K$, there exist constants $ \mu_K, \lambda_K $, and $ \nu_K $ such that the $ p $-adic valuation of the class number $ h_n $ of the degree $ p^n $ subextension $ K_n $ in $K_\infty/K$ satisfies
\[
\operatorname{ord}_p(h_n) = \mu_K p^n + \lambda_K n + \nu_K \; \text{for } n\gg 0.
\]

In~\cite{Maz72}, Mazur studied this question in the context of Tate--Shafarevich groups of abelian varieties with good ordinary reduction. If $ A/K $ is an abelian variety with good ordinary reduction at all primes above $ p $, and if the $ p $-primary Selmer group over the $\mathbb{Z}_p$-extension $ K_\infty $ is cotorsion as a module over the corresponding Iwasawa algebra, then one obtains an analogous asymptotic formula describing the growth of the $ p $-primary component of the Tate--Shafarevich group. More precisely, assume that $ \Sha(A/K_n)[p^\infty] $ is finite for all $ n $, there exist constants $ \mu_A, \lambda_A $, and $ \nu_A $ such that
\[
\operatorname{ord}_p(\Sha(A/K_n)[p^\infty]) = \mu_A p^n + \lambda_A n + \nu_A,
\]
for all $ n \geq 0 $ (see \cite{Maz72}, \cite[Theorem 1.10]{Gre99} and \cite[Theorem C]{Lee20}).

\vspace{1.3mm}

The supersingular reduction case is more involved, even for elliptic curves over the cyclotomic $\mathbb Z_p$-extension, as the classical approach does not apply directly. For an elliptic curve $ E/\mathbb{Q} $ with good supersingular reduction at $ p $, and the cyclotomic $\bbZ_p$-extension of $\bbQ$, the case where $a_p(E) = 0$ has been
studied by Kurihara \cite{kur02}, Kobayashi \cite{kob03} and Pollack \cite{Pol05}. 

\vspace{1.3mm}

Again over the cyclotomic $\Zp$-extension of $\bbQ$, Sprung \cite{Spr13, Spr17} studied this question for $p$-supersingular elliptic curves with $a_p(E)\neq 0$ and for abelian varieties of $\gl_2$-type. For related results concerning upper bounds on the growth of Bloch--Kato--Shafarevich--Tate groups associated to higher-weight modular forms, see \cite{LLZ17}. The growth of Mordell--Weil ranks and Tate--Shafarevich groups of elliptic curves over the cyclotomic $\bbZ_p$-extension of certain number fields has also been studied in \cite{LL22}. An explicit sufficient condition for the boundedness of Mordell--Weil ranks over the cyclotomic $\bbZ_p$-extension of number fields where $p$ is ramified was in \cite{LP20} for abelian varieties with good supersingular reduction at $p$.

\vspace{1.3mm}

Parallel developments have taken place for elliptic curves in the anticyclotomic setting. Let $K$ be an imaginary quadratic field and $E$ an elliptic curve defined over $\mathbb{Q}$. Assuming $p \geq 5$ and that the pair $(E, K)$ satisfies the Heegner hypothesis, {\c C}iperiani proved in \cite{Cip09} that the Tate--Shafarevich group of $E$ has trivial corank over the Iwasawa algebra associated with the anticyclotomic $\mathbb{Z}_p$-extension of $K$. A more refined formula describing the growth of these groups along the extension is given in \cite{matar21, llm23} by Matar and by Lei--Lim--M{\"{u}ller}, respectively. Further results concerning the inert primes and the elliptic curves with $a_p = 0$ over the anticyclotomic $\mathbb{Z}_p$-extension of $K$ have been obtained in \cite{IL24, bko24,muller25}; in the latter cases, the elliptic curve is assumed to have complex multiplication by an order in $K$ and $p \geq 5$.

\vspace{1.3mm}

Under certain hypotheses on the vanishing of the Mordell--Weil rank and the $p$-primary part of the Tate--Shafarevich group of an elliptic curve $E$ over a number field $K$, Iovita and Pollack \cite[Theorem 5.1]{ip06} established a formula describing the growth of Tate--Shafarevich groups over a $\mathbb{Z}_p$-extension of $K$ in which $p$ splits completely and each prime above $p$ is totally ramified in the $\bbZ_p$-extension.

\vspace{1.3mm}

The main objective of this article is to extend the aforementioned theorem of Iovita and Pollack in \cite{ip06} to the setting of abelian varieties of $\mathrm{GL}_2$-type over $\mathbb{Q}$ with real multiplication, arising from weight two newforms.

\subsection{Statement of the main theorem} Let $f$ be a weight two newform of level $N$ and trivial nebentype and let $F$ be the Hecke field generated by the Fourier coefficients $a_n$ of $f$. Note that since the nebentype is trivial, the Hecke field $F$ is totally real (see \cite[Proposition 3.4]{rib04}). Write $\ccO_F$ for the ring of integers of $F$. Let $A:=A_f$ be the associated $\gl_2$-type abelian variety over $\bbQ$ of dimension $g$ so that $\ccO_F\hookrightarrow \End(A)$ and 
\begin{equation*}
    \prod L(f^\sigma,s)=L(A,s),
\end{equation*}
where the product runs over all Galois conjugates $f^\sigma$ of $f$. Note that $A$ is self-dual. 

\vspace{1.3mm}
 
 Let $p$ be an odd prime, and suppose that $A$ has good \textit{supersingular} reduction at $p$. Let $K/\bbQ$ be a finite extension and $K_\infty/K$ be a $\bbZ_p$-extension. For an integer $n \geq 0$, we write $K_n$ for the unique subextension of $K_\infty$ such that $[K_n : K] = p^n$.

\vspace{1.3mm}

 The Galois group of $K_\infty$ over $K$ is denoted by $\Gamma$. In addition, write $\Gamma_n = \gal(K_\infty/K_n)$ and $G_n = \gal(K_n/K)$. We fix once and for all a topological generator $\gamma$ of $\Gamma$. For a prime $v$ of $F$ lying above $p$, let $F_v$ denote the completion of $F$ at $v$ with the integer ring $\ccO_v$.  Let $\Lambda_v$ denote the Iwasawa algebra $\ccO_v[[\Gamma]]=\displaystyle\varprojlim_n\ccO_v[G_n]$, which we shall identify with the power series ring $\ccO_v[[X]]$ by sending $\gamma-1$ to $X$. Let $(-)^\vee := \Hom(-,F_v/\ccO_v)$ denote the Pontryagin duality functor.

 \vspace{1.3mm}

We will assume the additional hypothesis on the splitting type of the prime $p$ in $K$ and $K_\infty$.

\begin{conj}\label{splittingtype}
    The prime $p$ splits completely in $K$ into $d:=[K:\bbQ]$ distinct primes, $\frap_1,\dots,\frap_d$. Moreover, each $\frap_i$ is totally ramified in $K_\infty$.
\end{conj}

In this article, we prove that the $\ccO_F$-rank of the Mordell–Weil group of $A$ over $K_\infty$ is bounded assuming that the signed Selmer groups $\Sel_v^{\bfs}(A/K_\infty)$ are $\Lambda_v$-cotorsion for all choices of $\bfs \in\{\sharp,\flat\}^d$ (see \S \ref{signedselmergroups}).  Furthermore, we derive an asymptotic formula for the growth of the $v$-parts of the Tate--Shafarevich groups $ \Sha(A/K_n )$ of $A$ in terms of the Iwasawa invariants of $\Sel_v^{\bfs}(A/K_\infty)^\vee$. The main result of this article is the following theorem.

\begin{thmx}\label{maintheorem} Let $A$ be an abelian variety of $\gl_2$-type over $\bbQ$ associated to a weight-two newform with trivial nebentype. Let $K$ be a number field and $p$ be an odd prime satisfying Hypothesis \ref{splittingtype}. Furthermore, assume that the signed Selmer groups $\Sel^\bfs_v(A/K_\infty)$ are cotorsion over $\Lambda_{v}$ (see Hypothesis \ref{cotorsionconjecture}).
Then the following statements hold.
\begin{itemize}
    \item [(i)] $\rank_{\ccO_F} A(K_n)$ is bounded independently of $n$. 
    \item [(ii)] Assume further that $v$-part of the $\ccO_F$-module $\Sha(A/K_n )$ is finite for all $n\geq 0$ and let $e_n := \ord_{v}\left(  \Sha(A/K_{n}) \right)$. Define \[\displaystyle r_\infty:=\sup_{n\ge0}\left\{ \rank_{\ccO_F} A(K_n)\right\}.\] 

 Then, for $n\gg 0$, we have
        \[e_n-e_{n-1}=  F_v(\bfs,n) +(p^n-p^{n-1})\mu_{v,\bfs} + e_v\lambda_{v,\bfs} -r_\infty, \]
         where $e_v$ denotes the ramification index of the extension $F_v/\Qp$, the sign $\bfs=(s_i)_{1\leq i\leq d}\in \{\sharp,\flat\}^d$ and the quantity $F_v(\bfs,n)$ are given in Proposition \ref{prop:kobayashirankofY''}, which depend on the parity of $n$; the quantities $\mu_{v,\bfs}$ and $ \lambda_{v,\bfs}$ denote the Iwasawa invariants of $\Sel_v^{\bfs}(A/K_\infty)^\vee$.
\end{itemize}
\end{thmx}

 Theorem~\ref{maintheorem} generalizes the works of Lei--Lim \cite[Theorems A and B]{LL22}, Kurihara~\cite[Theorem 0.1]{kur02}, Pollack~\cite[Theorem 1.1]{Pol05}, Sprung~\cite[Theorem 1.1]{Spr13} and Kobayashi~\cite[Theorem 1.4]{kob03} for elliptic curves over cyclotomic extensions. Furthermore, when $A$ is an elliptic curve, our theorem generalizes \cite[Theorem~5.1]{ip06}, removing their hypothesis (G) on Tamagawa number and the triviality of $\Sha(A/K)[p]$ and $A(K)/pA(K)$.

\begin{remark}
  When $K=\mathbb Q$,  the boundedness of $\rank_{\bbZ} A(K_n)$ has been established in \cite[Theorem 14.4]{kato04}. See also \cite[Theorem~3.4 and the discussion in \S3.3]{LP20} for an alternative approach to establish part (i) of Theorem~\ref{maintheorem}.
\end{remark}

 \begin{remark}
    Hypothesis~\ref{splittingtype} is also assumed in the setting of \cite{ip06}, where elliptic curves over $\bbQ$ with $a_p=0$ were studied. Our main contribution is to extend their construction of local points to the more general context of abelian varieties (see §\ref{constructionoflocalpoins}). These local points are then used to define signed Coleman maps (Proposition \ref{thm:signedcolemanmaps} via the results of \cite{bbl24}) and the associated signed Selmer groups. Similar constructions were given in \cite{bbl24} using the Perrin--Riou map, but our approach is more explicit and allows for a finer analysis of the joint image of the Coleman maps (Theorem \ref{thm:imageofsumofcolemanmaps}), which plays a crucial role in the proof of Theorem \ref{maintheorem}.
 \end{remark}

\subsection{Organization} In Section~\ref{sec:signedobjects}, we determine the Honda type of the formal groups associated with abelian varieties of $\gl_2$-type with real multiplication in our context and construct an explicit logarithm of the formal group associated with Lubin--Tate extensions of height one. This logarithm is used in Section~\ref{constructionoflocalpoins} to construct a system of local points, following the approach of \cite[\S 4.2]{ip06} and \cite[\S 8.4]{kob03}. These local points are then used in Section~\ref{sec:signedcolemanmaps} to define signed Coleman maps within the framework of \cite{bbl24}, and we prove a result describing the image of their direct sum. In Section~\ref{signedselmergroups}, we define the corresponding signed Selmer groups whose Iwasawa invariants appear in the statement of Theorem~\ref{maintheorem}. The section concludes with a review of certain global cohomology groups, under the assumption that the signed Selmer groups are cotorsion.  We recall the definition and basic properties of the Kobayashi rank following \cite[\S 10]{kob03} and \cite[\S 4]{LLZ17} in Section \ref{kobayashiranks}. In Section \ref{proofofthemaintheorem}, we establish a connection between Coleman maps and Kobayashi ranks, and show how this relationship enables us to study the growth of certain local modules. We conclude by bringing these ingredients together to prove Theorem \ref{maintheorem}. Relevant results on Honda type of formal groups are reviewed in Appendix~\ref{app} for the convenience of the reader.

\subsection{Outlook}
The principal novelty of the current article is the construction of local points to define signed Coleman maps, generalising the results of \cite{ip06} by using the formal groups associated with abelian varieties of $\gl_2$-type with real multiplication. 
As in the work of Iovita--Pollack, Hypothesis~\ref{splittingtype} allows us to work with Lubin--Tate extensions of height one when we localize at a prime above $p$. It may be possible to relax this hypothesis by working with relative Lubin--Tate formal groups instead. This would likely involve working with Honda theory for formal groups defined over a non-trivial extension of $\Qp$ developed in \cite{Hon70}. Since $[K_{\frap_i}:\Qp]$ can be $>1$ without assuming Hypothesis~\ref{splittingtype}, the method to calculate the Kobayashi ranks of modules that involve the $d\times d$ matrices for $d>1$ as developed in \cite{IL24} may also be useful.

Theorem~\ref{maintheorem} complements \cite[Theorem~A]{IL24}, where the growth of the $p$-primary Tate--Shafarevich groups of an elliptic curve over the anticyclotomic $\Zp$-extension of an imaginary quadratic field $K$ is studied, assuming that $p$ is inert in $K$. Unlike the present article, \cite{IL24} employs calculations that are specifically used to find the Kobayashi ranks of modules arise from $2\times 2$ matrices whose entries are elements of an Iwasawa algebra. It would be interesting to treat both cases within a single framework, in the spirit of the unified approach of \cite{BLV23}, where anticyclotomic Iwasawa main conjectures are studied in the $p$-split and $p$-inert settings simultaneously.

\subsection*{Acknowledgements}  We thank Meng Fai Lim for helpful discussions during the preparation of the article. The authors' research is supported by the NSERC Discovery Grants Program RGPIN-2020-04259 and RGPAS-2020-00096. EI is also partially supported by a postdoctoral fellowship from the Fields Institute.

\section{Lubin--Tate extensions and formal groups attached to abelian varieties}\label{sec:signedobjects}

Let $\{ k_n\}_{n\geq 0}$ with $\bbQ_p=k_0\subset k_1 \subset k_2 \subset \dots \subset k_\infty=\cup_n k_n$ be a tower of fields with $G_n=\gal(k_n/\bbQ_n)\cong\bbZ/p^n\bbZ$ such that $k_\infty$ is a totally ramified $\bbZ_p$-extension of $\bbQ_p$. Let $L_{n+1}:=k_n(\mu_p)$ for $n\geq 0$ and $L_\infty=\cup_n L_n=k_\infty(\mu_p)$. Here, if $M$ is a field, we denote by $M(\mu_p)$ the extension of $M$ obtained by adjoining the $p$-th roots of unity in some fixed algebraic closure of $M$. It then follows that $L_\infty$ is a $\bbZ_p^\times$-extension of $\bbQ_p$, and the group of its universal norms is generated by a uniformizer $\pi$ of $\bbZ_p$ such that $\ord_p\left( \tfrac{\pi}{p}-1\right)>0$. In particular, we may choose $\pi$ so that $k_n$ coincides with the completion of $K_n$ at a prime above $p$.

\vspace{1.3mm}

Following \cite[\S 4.1]{ip06}, we fix a one-dimensional Lubin--Tate formal group $\scrF$ of height-one over $\bbQ_p$, with parameter $\pi$ and the lift of Frobenius
\begin{equation*} 
    \fraf(X):= \pi X+ \sum_{i=2}^p \binom{p}{i} X^i \in \bbZ_p[[X]].
\end{equation*}
 For every $n$, we have $L_n=\bbQ_p(\scrF[\pi^n])$, where $\scrF[\pi^n]$ denotes the $\pi^n$-torsion of the formal group $\scrF$. 

\vspace{1.3mm}

We fix a $\ZZ$-basis of $\cO_F$, which induces an injective ring homomorphism
$\Phi:\ccO_F\longrightarrow M_{g\times g}(\bbZ)$. Let $C_r=\Phi(a_r)$ where $r\in\{q,q^2\}$ and $q$ is a prime. Define the formal Dirichlet series
\[
\sum_{n=1}^\infty C_n n^{-s}
=\prod_q(I_g-C_qq^{-s}+C_{q^2}q^{1-2s})^{-1},\]
where $C_n\in M_{g\times g}(\ZZ)$.

\begin{theorem}\label{thm:nart}
The formal group $\hatA$ of the abelian variety $A$ at $p$ is isomorphic over $\bbZ$ to the formal group whose logarithm is given by 
\[
\log_{\hatA}(\bfx)=\sum_{n\ge 1}\frac{C_n}{n}\bfx^n,
\]
 where $\bfx=(X_1,\dots,X_g)^t$ is a column vector of $g$ variables and $\bfx^{n}$ denotes $\left(X^{n}_1,\dots, X_g^{n} \right)^t$.
\end{theorem}
\begin{proof}
This is a consequence of \cite[Theorem 2.2]{deningernart}.
\end{proof}


Consider the sequence $\{x_k\}$ of $g\times g$ matrices defined by $x_{-1}=0$, $x_0=I_g$ and 
\begin{equation}\label{eq:recurence}
    px_k-C_p x_{k-1}+x_{k-2}=0 \text{ for } k\geq 1,
\end{equation}
where $C_p=\Phi(a_p)$. In particular, $p^k x_k$ has integral coefficients.

\begin{definition}
We define $\fraf^{(0)}(X)=X$ and if $k\ge1$ is an integer, we define $\fraf^{(k)}$ recursively by
\[
\fraf^{(k)}(X)=\fraf(\fraf^{(k-1)}(X)).
\]
Furthermore, we define
\[
\ell(\bfx)=\sum_{k\ge0}x_k\fraf^{(k)}(\bfx),
\]
where $\fraf^{(k)}(\bfx):=\left(\fraf^{(k)}(X_1),\ldots,\fraf^{(k)}(X_g) \right)^t$. 
\end{definition}

\begin{lemma}\label{lem:integral-xk}
    For all integers $k,t\ge0$, the entries of $x_k\fraf^{(t)}(\mathbf{x}+p\mathbf{y})-x_k\fraf^{(t)}(\mathbf{x})$ are defined in $p^{t+1-\lfloor k/2\rfloor}\Zp[X_1,\dots,X_g,Y_1,\dots, Y_g]$.
\end{lemma}
\begin{proof}
Since $A$ is supersingular at $p$, the entries of $C_p$ are divisible by $p$. Consequently, $x_1$ is defined over $\ZZ$. By induction, we can prove that the entries of $x_{k}$ are in $p^{-\lfloor k/2\rfloor}\ZZ$ using \eqref{eq:recurence}. 

The lemma would follow from the following inclusion
\[
\fraf^{(t)}(X+pY)-\fraf^{(t)}(X)\in p^{t+1}\Zp[X,Y],
\]
which can be proved using induction. Indeed, if $t=0$, the inclusion is immediate. Suppose that it holds for some $t\ge0$. Let us write \[\fraf^{(t)}(X+pY)=\fraf^{(t)}(X)+p^{t+1}S(X,Y).\]
Then
\begin{align*}
\fraf^{(t+1)}(X+pY)&=\fraf(\fraf^{(t)}(X)+p^{t+1}S(X,Y))\\
&=\pi\left(\fraf^{(t)}(X)+p^{t+1}S(X,Y)\right)+\sum_{i=2}^p\binom{p}{i}\left(\fraf^{(t)}(X)+p^{t+1}S(X,Y)\right)^i\\
&\equiv \pi\fraf^{(t)}(X)+\sum_{i=2}^p\binom{p}{i}\left(\fraf^{(t)}(X)\right)^i=\fraf^{(t+1)}(X)\mod p^{t+2}\Zp[X,Y]
\end{align*}
as $\pi\in p\Zp$. This gives the desired inclusion for all $t\ge0$ by induction.
\end{proof}

\begin{theorem}\label{thm:Q-split}

We have that $\hat\ell(\bfx,\bfy):= \ell^{-1}(\ell(\bfx)+\ell(\bfy))$ is a $g$-dimensional formal group over $\Zp$. Further, the formal groups $\hat{A}$ and $\hat\ell$ are isomorphic over $\Zp$.
\end{theorem}
\begin{proof}
In the light of Theorem~\ref{hondaisomorphismtheorem}, it is enough to show that the two formal groups have the same Honda type $u(T):=T^2-C_pT+p$. In particular, we need to show that
\begin{equation*}
    u*\ell\equiv u*\log_{\hat{A}}\equiv0\mod p\bbZ_p[[X]],
\end{equation*}
where  $u*\ell$ and $ u*\log_{\hat{A}}$ are given as in Definition \ref{definitionofu*l}. Note that $\fraf^{(k)}(X)\equiv X^{p^k} \mod p $. We then deduce that
\begin{align*}
    u*\ell&=\ell(\bfx^{p^2})-C_p\ell(\bfx^{p})+p\ell(\bfx)\\
    &=\sum_{k\ge0}x_k \fraf^{(k)}(\bfx^{p^2})-C_p\sum_{k\ge0}x_k \fraf^{(k)}(\bfx^p)+p\sum_{k\ge0}x_k \fraf^{(k)}(\bfx) \\
    &\equiv \sum_{k\ge2}(px_k-C_px_{k-1}+x_{k-2}) \bfx^{p^{k+2}} +p\bfx \mod p\\
    &\equiv 0 \mod p.   
\end{align*}
It then follows from Theorem \ref{hondaisomorphismtheorem} (i) that $\hat\ell$ is a formal group over $\Zp$.  The proof for $u*\log_{\hat{A}}$ is similar and relies on the fact that $C_{mp^k}=C_mC_{p^k}$ whenever $p\nmid m$.
\end{proof}

For the rest of this section, we identify these two formal groups.

\begin{lemma}\label{injectivityoflogarithm}
    The group homomorphism $\ell: \hatA(L_n)\longrightarrow {\widehat{\bbG}}^g_a(L_n)$ of formal groups is injective, where ${\widehat{\bbG}}^g_a$ is the additive formal group of dimension $g$.
\end{lemma}
\begin{proof}
Since the abelian variety is modular and has supersingular reduction at $ p$, the argument used in the proof of \cite[Lemma 4.4]{lei11}, which relies on the result given in \cite[Theorem~2.6]{edixhoven} and \cite[Proposition 4.1.4]{BLZ04} on the local residual representation of a $p$-non-ordinary modular form, can be applied to show that $ A(L_n)[p^\infty] = 0 $ for $ n \geq 0$. Since the kernel of $\ell$ consists of torsion points and $\hatA(L_n)$ has no prime-to-$p$ torsion (see \cite[Proposition C.2.5 and Theorem C.2.6]{HS00}), the lemma follows.
\end{proof}

\section{Construction of local points over Lubin--Tate extensions}\label{constructionoflocalpoins}

Choose a $\pi$-sequence $\{e_n\}_{n\geq 0}$ in $L_\infty$, i.e. $e_n\in \scrF[\pi^n]\setminus \scrF[\pi^{n-1}]$ such that $\fraf(e_n)=e_{n-1}$ for all $n\geq 1$ and $e_0=0$. 

Let $u_1\in\ZZ^{\oplus g}$ be the column vector that represents $1\in\cO_F$ with respect to the basis chosen for the ring homomorphism $\Phi$. In particular, the orbit of $u_1$ under the action of $\cO_F$ (induced by $\Phi$) is $\ZZ^{\oplus g}$. After a change of basis if necessary, we can take $u_1=(1,\dots,1)^t$. Indeed, we can choose an integral basis of $F$ that is of the form $\{1,\alpha_2,\dots,\alpha_g\}$. Then $\{1-\sum_{i=2}^g\alpha_i, \alpha_2,\dots,\alpha_g\}$ is another integral basis of $F$. The vector representing $1$ with respect to this basis is then equal to  $u_1$. 

Let $\epsilon\in \left(p\bbZ_p\right)^g$ be such that
\[
\ell(\epsilon)= (1+p-C_p)^{-1}\left(pu_1\right).
\]
Note that we have written $1+p$ for the $g\times g$ scalar matrix whose diagonal entries are all equal to $1+p$. As $a_p$ is a non-unit of $\ccO_F$, the matrix $C_p$ is $p$-adically nilpotent. Thus, $1+p-C_p$ is $p$-adically unipotent and belongs to $\gl_g(\bbZ_p)$. We regard $\epsilon$ as an element of $\hatA(\Qp)$.

\begin{definition}
    We define $c_0=\epsilon\in \hatA(\Qp)$, and
\begin{equation*}
    c_n:=\epsilon [+]_{\hatA} (e_n u_1)\in\hatA(L_n),\; \text{for all}\; n\geq 1,
\end{equation*}
where $[+]_{\hatA}$ is the addition of $\hatA$.

We furthermore define $d_n$ to be the image of $c_{n+1}$ under the trace map $\hatA(L_{n+1})\rightarrow \hatA(k_n)$ for $n\ge0$.
\end{definition}

  Fix a prime $v$ of $F$ lying above $p$ and let $F_v$ denote the completion of $F$ at $v$ with the integer ring $\ccO_v$. Note that $\Phi$ induces an $\cO_v$-action on $\hatA(M)$ if $M$ is a finite extension of $\Qp$. Under this action, $c_0=\epsilon$ is a generator of the $\ccO_v$-module $\hatA(\Qp)$. Given finite extensions $M_1\supset M_2\supset \Qp$, let $\tr_{M_1/M_2}$ be the trace map $M_1\longrightarrow M_2$. By an abuse of notation, we use the same notation to denote the trace map $\hatA(M_1)\longrightarrow \hatA(M_2)$ with respect to the group law on $\hatA$.

\begin{theorem}\label{thm:QsystemforT}
 For $n\geq 0$, there exists $d_n\in \hatA(k_n)$ such that:
    \begin{enumerate}[label=(\roman*)]
        \item For $n\ge1$, $\tr_{n+1/n}\left( d_{n+1} \right)= C_pd_n- d_{n-1}$.
        \item  $\tr_{1/0}\left(d_1\right)=\left(C_p-(C_p-2)^{-1}(p-1)\right)\cdot d_0$.
 \end{enumerate}
\end{theorem}
\begin{proof} Our proof follows closely the calculations in \cite[\S4]{ip06}. Let $n\ge2$. Note that the minimal polynomial of $e_n$ over $L_{n-1}$ is $\fraf(X)-e_{n-1}$; we have $\tr_{L_n/L_{n-1}}(e_n)=-p$. Thus,
\begin{align*}
       &\ \ \ \ \ell\left(\tr_{L_n/L_{n-1}}\left( c_{n}\right)\right)\\
       &=\tr_{L_n/L_{n-1}}\left( \ell \left(c_{n}\right)\right)\\
       &=  \tr_{L_n/L_{n-1}}\left((1+p-C_p)^{-1}\left(pu_1\right)+\sum_{k\geq 0} x_k \left( \fraf^{(k)}(e_{n})u_1\right)\right)  \\
        &=  (1+p-C_p)^{-1}\left(p^2u_1\right)-pu_1+p\sum_{k\geq1} x_k  \left( e_{n-k}u_1\right) \\
        &=  C_p(1+p-C_p)^{-1}(pu_1)-(1+p-C_p)^{-1}(pu_1)+\sum_{k\geq1} (C_px_{k-1}-x_{k-2})  \left( e_{n-k}u_1\right) \\
           &=  C_p(1+p-C_p)^{-1}(pu_1)-(1+p-C_p)^{-1}(pu_1)+\\
         &\qquad \qquad \qquad  \qquad  \qquad  
 \qquad   C_p\sum_{k\geq1} x_{k-1}  \left( e_{n-k}u_1\right) - \sum_{k\geq1} x_{k-2}  \left( e_{n-k}u_1\right)\\
          &=  C_p(1+p-C_p)^{-1}(pu_1)-(1+p-C_p)^{-1}pu_1+\\
         &\qquad \qquad \qquad  \qquad  \qquad  
 \qquad  A_p\sum_{k\geq0} x_{k}  \left( e_{n-1-k}u_1\right) - \sum_{k\geq0} x_{k}  \left( e_{n-2-k}u_1 \right)\\
         &=C_p\ell(c_{n-1})-\ell(c_{n-2}).
    \end{align*}
Similarly, as the minimal polynomial of $e_1$ over $\Qp$ is $\fraf(X)/X$, we have $\tr_{L_1/\Qp}(e_1)=-p$. This allows us to deduce
\begin{equation*}
     \ell(\tr_{L_1/\Qp}(c_1))=(C_p-2)\ell(c_{0}).
\end{equation*}
Define $d_n:=\tr_{L_{n+1}/k_n}(c_{n+1})$ for $n\geq 1$. Taking norms from $L_{n+1}$ to $k_n$, we deduce that 
\begin{equation*}
    \tr_{k_n/k_{n-1}}(d_{n})=C_p d_{n-1}-d_{n-2}, \text{ for } n\geq 2,
\end{equation*}
 and $d_0=(C_p-2)c_0$. This gives
\[
\tr_{k_1/\Qp}(d_1)=C_p d_0-(C_p-2)^{-1}(p-1)d_0, 
\]
as required.
\end{proof}

\begin{prop}\label{pointsgeneratetheformalgroup}
    For $n\geq 1$,  $d_n$ and $d_{n-1}$ generate ${\hatA}(k_n)$ as an $\cO_v[\gal(k_n/\bbQ_p)]$-module.
\end{prop}

\begin{proof} We have chosen $c_0$ to be an $\cO_v$-generator of $\widehat{A}(L_0)$. Suppose that $c_{n}$ and $c_{n-1}$ generate $\hatA(L_n)$ as an $\cO_v[\gal(L_n/\bbQ_p)]$-module. We show that this implies $c_{n}$ and $c_{n+1}$ generate $\hatA(L_{n+1})$ as an $\cO_v[\gal(L_{n+1}/\bbQ_p)]$-module. Since $c_{n-1}= C_pc_n-\tr_{L_{n+1}/L_n}\left( c_{n+1} \right) $, we have
\[
\hatA(L_n)\subseteq \cO_v[\gal(L_n/\bbQ_p)] c_n+\cO_v[\gal(L_{n+1}/\bbQ_p)]c_{n+1}.
\]
Therefore, it suffices to show the following equality:
\begin{equation}\label{eq:quotients}
    \frac{\cO_v[\gal(L_{n+1}/\bbQ_p)]c_{n+1}+\hatA(L_n)}{\hatA(L_n)}=\hatA(L_{n+1})\big/\hatA(L_n).
    \end{equation}
As $\ell$ is injective on $\hatA(L_{n+1})$ by Lemma~\ref{injectivityoflogarithm}, it suffices to show that $\ell(c_{n+1})$ generates the quotient $\ell(\fram_{n+1}^{\oplus g})/\ell(\fram_{n}^{\oplus g})$. 

Let $\mathbf{y}=(y_1,\ldots,y_g)\in \hatA(L_{n+1})$. Then \cite[Proposition~4.4]{ip06} tells us that each $y_i$ can be expressed as $\sum_{\sigma\in\gal(L_{n+1}/L_n)} b_\sigma e_{n+1}^\sigma+c$, where $b_\sigma\in\Zp$ and $c\in\fram_n $. Since $\fraf^{(s)}(X)\equiv X^{p^s} \mod p $ for all $s\ge1$, we have
\[
\fraf^{(s)}(y_i)\equiv \sum_{\sigma\in\gal(L_{n+1}/L_n)} b_\sigma (e_{n+1}^\sigma)^{p^s}+c^{p^s}\equiv\sum_{\sigma\in\gal(L_{n+1}/L_n)} b_\sigma \fraf^{(s)}(e_{n+1}^\sigma)+\fraf^{(s)}(c)\mod p\cO_{L_{n+1}}.
\]
Furthermore, $\fraf^{(s)}(e_{n+1}^\sigma)\in\fram_n$. Therefore,
for all $k\ge1$, setting $s=k-\lfloor k/2\rfloor$ and applying Lemma~\ref{lem:integral-xk} with $t=\lfloor k/2\rfloor$, we deduce that
\[
x_k\cdot\fraf^{(k)}(\mathbf{y})=x_k\cdot\fraf^{(t)}(\fraf^{(s)}(\mathbf{y}))\in (\fram_{n+1}+L_n)^{\oplus g}.
\]
Consequently, we deduce that
\[
\ell\left( \hatA(L_{n+1})\right)\subseteq (\fram_{n+1}+L_n)^{\oplus g}.
\]
This implies that
\[
\ell(\fram_{n+1}^{\oplus g})/\ell(\fram_{n}^{\oplus g})
\cong \fram_{n+1}^{\oplus g}/\fram_{n}^{\oplus g}. 
\]
Recall that $c_{n+1}$ generates $\fram_{n+1}/\fram_{n}$ as a $\Zp[\gal(L_{n+1}/\Qp)]$-module and $\ell(c_{n+1})\equiv c_{n+1}u_1\mod\fram_n^{\oplus g}$. Furthermore, as an $\cO_v$-module, $u_1$ generates $\Zp^{\oplus g}$. Hence, we deduce that $\ell(c_{n+1})$ generates $\ell(\fram_{n+1}^{\oplus g})/\ell(\fram_{n}^{\oplus g})$ as an $\cO_v[\gal(L_{n+1}/\Qp)]$-module.

    \vspace{1.3mm}

    Since $[L_{n+1} : k_n]$ is prime to $p$, the map $\tr_{L_{n+1}/k_n} : \hatA (L_{n+1}) \longrightarrow \hatA(k_n)$ is surjective. Indeed, we have that the space $\hatA(L_{n+1})$ decomposes as a direct sum of eigenspaces:
    \begin{equation*}
        \hatA(L_{n+1})=\bigoplus_{\chi}\hatA(L_{n+1})^{\chi},
    \end{equation*}
    where $\chi$ runs over the distinct characters of $\gal(L_{n+1}/k_n)$, and $\hatA(L_{n+1})^{\chi}$ is the eigenspace corresponding to the character $\chi$. For $\chi\neq \mathbf{1}$ and $a\in \hatA(L_{n+1})^\chi$, we have
    \begin{equation*}
        \tr_{L_{n+1}/k_n}(a)=\sum_{\sigma\in \gal(L_{n+1}/k_n)} \chi(\sigma)a=\left(\sum_{\sigma\in \gal(L_{n+1}/k_n)} \chi(\sigma)
        \right) a=0.
    \end{equation*}
    On the other hand, for $\chi=\mathbf{1}$ and  $a\in \hatA(k_{n})$, we have
    \begin{equation*}
        \tr_{L_{n+1}/k_n}(a)=\sum_{\sigma\in \gal(L_{n+1}/k_n)} \chi(\sigma)a=(p-1)a.
    \end{equation*}
    Since $p-1$ is prime to $p$, the multiplication by $p-1$ map is surjective on $\hatA(k_n)$. Thus, the trace map is surjective on $\hatA(L_{n+1})$.
    
    Since $c_n$ and $c_{n-1}$ generate $\hatA(L_{n+1})$ as a $\ccO_v[\gal(L_{n+1}/\bbQ_p)]$-module, we have that $d_n$ and $d_{n-1}$ generate $\hatA(k_n)$ as an $\ccO_v[\gal(k_{n}/\bbQ_p)]$-module for $n \geq 1$. 
\end{proof}

\section{Signed Coleman maps}\label{sec:signedcolemanmaps} 
  Write $T_v$ for the $v$-adic Tate module of $A$, which is a free $\ccO_v$-module of rank $2$ equipped with a continuous $G_\bbQ$-action. Write $H_{\bff}^1(k_n,T_v)$ for the image of the Kummer map $\hatA(k_{n}) \longrightarrow H^1(k_{n},T_v)$, and similarly $H_{\bff}^1(L_n,T_v)$ for $\hatA(L_{n})$.

   \begin{prop}\label{thm:QsystemforTv}
       There exists a primitive $Q$-system $\bbd=(d_n)_{n\geq 0}$  for $T_v$ in the sense of \cite[Definition 4.1]{bbl24}; that is, $\bbd=(d_n)_{n\geq 0}$ satisfies the following properties:
       \begin{itemize}
           \item[(i)]  $d_0 \in\hatA(\Qp)\setminus  p\hatA(\Qp) $;
           \item[(ii)] $\tr_{k_1/\Qp}(d_1)\in \hatA(\Qp)\setminus  p\hatA(\Qp)$;
           \item[(iii)] $\tr_{k_{n+1}/k_n}(d_{n+1})= C_p d_n- d_{n-1} $ for all $n\geq 1$.
       \end{itemize}
\end{prop}
\begin{proof}
 The assertion (i) follows from the definition $d_0=(C_p-2)\epsilon$. The assertions (ii) and (iii) follow from Theorem \ref{thm:QsystemforT}. 
\end{proof}

\vspace{1.3mm}

 For $n\geq 0$, let $\omega_n=(1+X)^{p^n}-1=\displaystyle\prod_{0\le i\le n}\Phi_i$, where $\Phi_{i}$ is the $p^i$-th cyclotomic polynomial in $1+X$. Set $\Lambda_{v}:=\varprojlim_n \ccO_v[G_n]$ so that $\Lambda_{v,n}:= \Lambda_{v} \big/(\omega_n)=\ccO_v[G_n]$, where we have $G_n=\gal(k_n/\bbQ_n)$. Define the following matrices
         \[
        \ccC_i(X):=\begin{bmatrix}
            a_p& 1\\ -\Phi_i&0
        \end{bmatrix} 
        \]
        and
        \[
        \ccH_n:=\ccC_n(X)\dots \ccC_1(X).\]
        Let $n\geq1$ be an integer. Let  \begin{equation*}
     \langle-,-\rangle_{n}: \hatA(k_{n}) \times  H^1(k_{n},T_v) \longrightarrow \ccO_v
 \end{equation*}
 denote the local pairing ($A$ is self-dual).  We define $H_\iw^i(k_{\infty},T_v)=\displaystyle\varprojlim_n H^i(k_{n},T_v)$, where the transition maps are given by the corestriction maps. Suppose that $\bbd=(d_n)_{n\geq 0}$ is a primitive $Q$-system for $T_v$ given as in Proposition \ref{thm:QsystemforTv}. Define
 \begin{align}
      \col_{v,n}: \HIw(k_\infty,T_v)&\longrightarrow \Lambda_{v,n}\notag\\
      z&\mapsto \sum_{\sigma\in G_n} \langle z^{\sigma^{-1}},d_n\rangle_{n}\cdot\sigma.\label{eq:colvn}
 \end{align}

\begin{theorem}\label{thm:signedcolemanmaps}
    There exist unique $\Lambda_v$-morphisms (called the $\sharp/\flat$-Coleman maps)
\begin{equation*}
      \col^{\sharp}_{v}, \col^{\flat}_{v}: \HIw(k_\infty,T_v)\longrightarrow \Lambda_{v}
\end{equation*}
such that
\begin{equation*}
    \ccH_n \begin{bmatrix}
        \col^{\sharp}_{v}(z)\\ \col^{\flat}_{v}(z)
    \end{bmatrix} \equiv \begin{bmatrix}
        \col_{v,n}(z)\\ -\xi_{n-1}\col_{v,n-1}(z)
    \end{bmatrix} \mod \omega_n,
\end{equation*}
where $\xi_{n-1}:\Lambda_{v,n-1}\longrightarrow \Lambda_{v,n}$ is the norm map.
Further, the Coleman maps $\col^{\sharp}_v, \col^{\flat}_v$ are surjective onto $\Lambda_{v}$.
\end{theorem}
\begin{proof}
    This is Corollary 4.5 and Proposition 4.7 in \cite{bbl24} applied to our case.
\end{proof}

\begin{prop}\label{prop:image}
    Let $J_{v}:=\{ (g_1,g_2)\in \Lambda_{v}^{\oplus 2} \mid (p-1)g_1(0)=(2-a_p)g_2(0) \}$. Then we have
    \begin{equation*}
        \image(\col_{v}^{\sharp}\oplus \col_{v}^{\flat})\subseteq J_{v}.
    \end{equation*}
\end{prop}
\begin{proof}
On taking $n=1$ in Theorem~\ref{thm:signedcolemanmaps}, we have
\begin{equation}
\begin{bmatrix}
    a_p&1\\ -\Phi_1&0\end{bmatrix}\begin{bmatrix}
        \col^{\sharp}_{v}(z)\\ \col^{\flat}_{v}(z)
    \end{bmatrix} \equiv \begin{bmatrix}
        \col_{v,1}(z)\\ -\xi_{0}\col_{v,0}(z)
    \end{bmatrix} \mod \omega_1.
\label{eq:cong-omega1}    
\end{equation}
Combined with Theorem~\ref{thm:QsystemforT}(ii), we deduce that
\[\begin{bmatrix}
    a_p&1\\ -p&0\end{bmatrix}\begin{bmatrix}
        \col^{\sharp}_{v}(z)\\ \col^{\flat}_{v}(z)
    \end{bmatrix} \equiv \begin{bmatrix}
        \left(a_p-\frac{p-1}{a_p-2}\right)\col_{v,0}(z)\\ -p\col_{v,0}(z)
    \end{bmatrix} \mod X
\]
Thus,
\[
a_p\col_v^\sharp(z)+\col_v^\flat(z)\equiv\left(a_p-\frac{p-1}{a_p-2}\right)\col_{v,0}(z),\quad \col_v^\sharp(z)\equiv\col_{v,0}(z)\mod X.
\]
Combining these two congruences gives
\[
(p-1)\col_v^\sharp(z)\equiv (2-a_p)\col_v^\flat(z)\mod X.
\]
In particular, $\image(\col_{v}^{\sharp}\oplus \col_{v}^{\flat})\subseteq J_v$, as required.\end{proof}

\begin{theorem}\label{thm:imageofsumofcolemanmaps}
    We have the equality $\image(\col_{v}^{\sharp}\oplus \col_{v}^{\flat})= J_{v}.$
\end{theorem}

\begin{proof}
Let $\bbF_v$ denote the residue field of $\cO_v$ and fix a uniformizer $\varpi$. We write $\Omega=\Lambda_v/(\varpi)$.
Let $\uCol:\HIw(k_\infty,T_v)\longrightarrow\Omega^{\oplus 2}$ denote the map given by $\col_v^\sharp\oplus\col_v^\flat$ modulo $\varpi$ and write $\underline{J_v}$ for the image of $J_v\subset \Omega^{\oplus 2}$ under the reduction map. It follows immediately from Proposition~\ref{prop:image} that 
\[
\image(\uCol)\subseteq\underline{J_v}.
\]
We claim that the theorem would follow from the equality $\image(\uCol)=\underline{J_v}$. Indeed, if this equality holds, given any $x\in J_v$, we can write $x=y+pz$ for some $y\in\image(\col_{v}^{\sharp}\oplus \col_{v}^{\flat})$ and $z\in\Lambda_v^{\oplus 2}$. This implies that $x-y=pz\in J_v\bigcap p\Lambda_v^{\oplus 2}$. A direct calculation shows that this intersection is $pJ_v$. Consequently, we deduce that $J_v=\image(\col_{v}^{\sharp}\oplus \col_{v}^{\flat})+pJ_v$. Nakayama's lemma then implies $J_v=\image(\col_{v}^{\sharp}\oplus \col_{v}^{\flat})$.

We verify the equality $\image(\uCol)=\underline{J_v}$ in the remainder of the proof. Since $\Omega^{\oplus 2}/\underline{J_v}\cong\bbF_v$, it is enough to show that the cokernel of $\uCol$ is isomorphic to $\bbF_v$. We prove this following a similar strategy as \cite[the proof of Proposition~1.2]{KP}.

Let $\underline{\hatA(k_n)}^\star$ denote the $\bbF_v[G_n]$-submodule of $\underline{\hatA(k_n)}:=\hatA(k_n)/(\varpi)$ generated by $d_n$. We can check that $\underline{\hatA(k_n)}^\star\bigcap\underline{\hatA(k_{n-1})}^\star=\hatA(\Qp)$ for all $n\ge1$. Combined with Proposition~\ref{pointsgeneratetheformalgroup} gives the following short exact sequence
\begin{equation}
0\rightarrow \underline{\hatA(\Qp)}\rightarrow \underline{\hatA(k_n)}^\star\oplus \underline{\hatA(k_{n-1})}^\star\rightarrow \underline{\hatA(k_n)}\rightarrow 0
\label{eq:plus-minus-split}    
\end{equation}
as in \cite[Proposition~8.12]{kob03}.

For $\natural\in\{\sharp,\flat\}$, let $\ccP^\natural$ be the image of $\ker (\col_v^\natural)$ in $\HIw(k_\infty,T_v)/(\varpi)$. Since the matrix $\ccC_i(X)$ is congruent to $\begin{bmatrix}
    0&1\\ -\Phi_i&0
\end{bmatrix}$  modulo $\varpi$, it follows from Theorem~\ref{thm:signedcolemanmaps} that 
\[
\ccP^\sharp=\varinjlim \left(\underline{\hatA(k_{2n})}^\star\right)^\perp,\quad\ccP^\flat=\varinjlim \left(\underline{\hatA(k_{2n+1})}^\star\right)^\perp,
\]
where $(-)^\perp$ denotes the orthogonal complement under the pairing
\[
\HIw(k_\infty,T_v)/(\varpi)\times \hatA(k_\infty)/(\varpi)\rightarrow\bbF_v
\]
induced from the local Tate duality
\[
\HIw(k_\infty,T_v)\times \left(\hatA(k_\infty)\otimes_{\ccO_F} F_v/\cO_v\right)\rightarrow F_v/\cO_v
\]
after identifying $\hatA(k_\infty)\otimes_{\ccO_F} F_v/\cO_v$ with $\HIw(k_\infty,A[v^\infty])$ (which we can do since $\HIw(k_\infty,A[v^\infty])/\hatA(k_\infty)\otimes_{\ccO_F} F_v/\cO_v$ is the Pontryagin dual of the universal norm $\displaystyle\varprojlim_n \hatA(k_n)$, which is 0 when $A$ is supersingular at $p$; see \cite[Theorem~1]{schneider}).

Letting $n\rightarrow \infty$ in \eqref{eq:plus-minus-split} gives the following short exact sequence
\[0\rightarrow \underline{\hatA(\Qp)}\rightarrow \left(\ccP^\sharp\right)^\perp\oplus \left(\ccP^\flat\right)^\perp\rightarrow \HIw(k_\infty,A[v^\infty]) \rightarrow 0.
\]
By duality, we deduce that the cokernel of $\uCol$ is isomorphic to $\underline{\hatA(\Qp)}^\vee\cong\bbF_v$, as desired.
\end{proof}

\begin{remark}
    In prior works on supersingular abelian varieties, the anti-block diagonal property of the logarithmic matrices related to the Coleman maps is sometimes assumed explicitly, which is equivalent to the hypothesis that $a_p = 0$ in the setting of $\gl_2$-type abelian varieties (see \cite{LP20,dionray}). Our approach allows us to relax this hypothesis, where we study the joint image of the signed Coleman maps. This plays a crucial role in analyzing the growth of the $v$-primary part of the Tate--Shafarevich group via auxiliary $\ccO_v$-modules.
\end{remark}

\section{Signed Selmer groups}\label{signedselmergroups}  Recall that $K_\infty$ is a $\bbZ_p$-extension of $K$ such that each prime $\frap_1,\ldots,\frap_d$ of $K$ above $p$ is totally ramified in $K_\infty$, and that $K_n \subset K_\infty$ denotes the unique subextension such that $[K_n : K] = p^n$, for an integer $n \geq 0$.

\vspace{1.3mm}

For $i\in\{1,\dots, d\}$, write $K_{\frap_i}$ for the completion of $K$ at $\frap_i$, which is isomorphic to $\bbQ_p$, and $K_{\infty,\frap_i}$ for the $\bbZ_p$-extension of $K_{\frap_i}$ corresponding to $k_\infty$. Let $K_{n,\frap_i}$ denote the unique subextension of $K_{\infty,\frap_i}/K_{\frap_i}$ that is of degree $p^n$.  For every place $v$ of $F$ lying above $p$, Theorem \ref{thm:signedcolemanmaps} gives the signed Coleman maps
\begin{equation*}
     \col^{\bullet}_{v,i}: \HIw(K_{\infty,\frap_i},T_v)\longrightarrow \Lambda_{v},
\end{equation*}
where $\bullet\in \{ \sharp,\flat\}$ and $\frap_i\in \{\frap_1,\ldots,\frap_d\}$.

\begin{definition}
    For $n\geq 0$, $\bullet\in\{\sharp,\flat \}$ and  $i\in \{1,\ldots,d\}$, define $H^1_\bullet(K_{n,\frap_i},T_v)\subset H^1(K_{n,\frap_i},T_v)$ to be the image of $\ker \col_{v,i}^{\bullet}$ under the natural projection $\HIw(K_{\infty,\frap_i},T_v)\longrightarrow H^1(K_{n,\frap_i},T_v)$.

    \vspace{1.3mm}

  Fix a uniformizer $\varpi$ of $F_v$. Define $H^1_\bullet(K_{n,\frap_i},A[\varpi^\infty])\subset H^1(K_{n,\frap_i},A[\varpi^\infty])$ to be the orthogonal complement of $H^1_\bullet(K_{n,\frap_i},T_v)$ under the local Tate pairing 
       \begin{equation*}
      H^1(K_{n,\frap_i},T_v)\times  H^1(K_{n,\frap_i},A[\varpi^\infty]) \longrightarrow F_v/\ccO_v.
      \end{equation*}
\end{definition}

\vspace{1.3mm}

Let $\lambda$ be a prime of $K_n$ lying above a rational prime $\ell$. Write $H^1_{\bff}(K_{n,\lambda}, A[\varpi^\infty])\subset H^1(K_{n,\lambda}, A[\varpi^\infty])$ for the image of the Kummer map $A(K_{n,\lambda})\otimes_{\ccO_F} F_v/\ccO_v \longrightarrow H^1(K_{n,\lambda}, A[\varpi^\infty])$.

\vspace{1.3mm}

The singular quotient at $\ell$ is given by
\begin{equation*}
    H^1_{/\bff}(K_{n,\ell}, A[\varpi^\infty]):=\bigoplus_{\lambda\mid \ell}\dfrac{H^1(K_{n,\lambda}, A[\varpi^\infty])}{H^1_{\bff}(K_{n,\lambda}, A[\varpi^\infty])},
\end{equation*}
where the direct sum runs over all primes of $K_n$ above $\ell$.

\vspace{1.3mm}

\begin{definition}\label{def:pmSel}
   For $\bfs:=(s_i)_{1\leq i\leq d} \in\{\sharp,\flat\}^d$, we define the signed Selmer group of $A$ over $K_n$ by
    \begin{align*}
        \Sel_v^\bfs(A/K_n):=\ker\Bigg(H^1(K_n,A[v^\infty])\longrightarrow \prod_{\ell\nmid p} H^1_{/\bff}&(K_{m,\ell}, A[\varpi^\infty])  \\ &\times  \prod_{i=1}^d \dfrac{H^1(K_{n,\frap_i}, A[\varpi^\infty])}{H^1_{s_i}(K_{n,\frap_i}, A[\varpi^\infty])} \Bigg).
    \end{align*}
    Further, define $\Sel^\bfs_v(A/K_\infty)=\varinjlim_n \Sel^\bfs_v(A/K_n)$.
\end{definition}

\vspace{1.3mm}

If $\Sel_{v}(A/K_\infty)$ denotes the classical $v^\infty$-Selmer group, then we have that $\Sel^\bfs_v(A/K_\infty)\subset \Sel_{v}(A/K_\infty)$. It follows from \cite[Theorem 4.5]{man71} that $\Sel_{v}(A/K_\infty)$ is cofinitely generated over $\Lambda_{v}$. Hence, so is $\Sel^\bfs_v(A/K_\infty)$. We consider the following hypothesis.

\vspace{1.3mm}

\begin{conj}\label{cotorsionconjecture}
    For all choices of $\bfs \in\{\sharp,\flat\}^d$, the signed Selmer group $\Sel^\bfs_v(A/K_\infty)$ is cotorsion over $\Lambda_{v}$.
\end{conj}

\begin{remark}\label{cotorsionresult}
Let $E/\bbQ$ be an elliptic curve and $p$ be an odd prime where $E$ has good supersingular reduction with $a_p(E) = 0$. Under the assumption that $\corank_{\Zp} \Sel(E/K_n)$ is bounded, \cite[Corollary 7.7]{ip06} implies that Hypothesis \ref{cotorsionconjecture} holds.
\end{remark}

For the remainder of the article, we assume that Hypothesis \ref{cotorsionconjecture} holds.

\begin{definition}
 For $\bfs \in\{\sharp,\flat\}^d$, we write $\mu_{v,\bfs}$ and $\lambda_{v,\bfs}$ for the $\mu$- and $\lambda$-invariants of the torsion $\Lambda$-module $\Sel^\bfs_v(A/K_\infty)^\vee$.
  \end{definition}

\vspace{1.3mm}

Let $\Sigma$ denote a fixed finite set of primes of $K$ containing $p$, the ramified primes of $K/\bbQ$, the archimedean place, and all the bad reduction primes of $A$. Write $K_\Sigma$ for the maximal algebraic extension of $K$ which is unramified outside $\Sigma$. For any (possibly infinite) extension $K \subseteq K' \subseteq K_\Sigma$, write $G_\Sigma(K') = \gal(K_\Sigma/K')$. 

\vspace{1.3mm}
 
For $i\in\{1,2\}$, we define $H^i_{\iw,\Sigma}(K_\infty,T_v)=\displaystyle\varprojlim H^i(G_\Sigma(K_n),T_v)$, where the transition maps are given by the corestriction maps. Note that $H^i_{\iw,\Sigma}(K_\infty,T_v)$ is independent of the choice of $\Sigma$ (see \cite[Corollary B.3.5]{Ru00}, \cite[Lemma 5.3.1]{MR04} and \cite[Proposition 7.1]{kob03}). We will drop the subscript $\Sigma$ from the notation for simplicity and simply write $H^i_{\iw}(K_\infty,T_v)$.

 \vspace{1.3mm}

 We recall the definition of fine Selmer groups:
\begin{definition}
    For $0\leq n \leq \infty$, we define the fine Selmer group  
        \begin{equation*}
     \Sel_v^0(A/K_n):=\ker\left( \Sel_{v}(A/K_n)\longrightarrow H^1(k_n, A[\varpi^\infty]) \right).
    \end{equation*}

Equivalently, we have
 \begin{equation*}
     \Sel_v^0(A/K_n):=\ker\left( H^1(G_\Sigma(K_n),A[\varpi^\infty])\longrightarrow \prod_{\fraq_n\in \Sigma(K_n)} H^1(K_{n,\fraq_n}, A[\varpi^\infty]) \right),
    \end{equation*}
where $\Sigma(K_n)$ denotes the set of primes of $K_n$ lying above $\Sigma$.  

\end{definition}

The Pontryagin duals of $\Sel_{v}(A/K_\infty)$, $\Sel^\bfs_v(A/K_\infty)$ and  $ \Sel^0_v(A/K_\infty)$ are denoted by $\ccX_v(A/K_\infty)$, $\ccX^\bfs_v(A/K_\infty)$ and $\ccX^0_v(A/K_\infty)$, respectively.  By \cite[Proposition A.3.2]{Per00}, for $\bfs=(s_i)\in \{\sharp,\flat \}^d$, we have the following Poitou--Tate exact sequence
 \begin{equation}\label{eqn:poitoutate}
     H^1_{\iw}(K_\infty,T_v) \longrightarrow  \prod_{i=1}^d \dfrac{H^1_\iw(K_{\infty,\frap_i}, T_v)}{H^1_{s_i}(K_{\infty,\frap_i}, T_v)}   \longrightarrow\ccX^\bfs_v(A/K_\infty) \longrightarrow \ccX^0_v(A/K_\infty) \longrightarrow 0,
\end{equation}
and, for any $n\geq 0$.
\begin{multline}\label{eqn:poitoutateovern}
       H^1(G_\Sigma(K_n),T_v) \longrightarrow   \prod_{i=1}^d \dfrac{H^1(K_{n,\frap_i}, T_v)}{A(K_{n,\frap_i})\otimes_{\ccO_F}\ccO_v}  \longrightarrow \ccX_v(A/K_\infty)\\ \longrightarrow \ccX^0_v(A/K_\infty) \longrightarrow 0.
\end{multline}

We conclude this section with the following statement on the $\Lambda_v$-module structure of $H^i_\iw(K_\infty,T_v)$.

\vspace{4.3mm}

\begin{prop}\label{consequenceofcotorsionness}
    If Hypothesis~\ref{cotorsionconjecture} holds,  the following statements are valid:
    \begin{enumerate}[label=(\roman*)]
        \item $H^2_\iw(K_\infty,T_v)$ is a torsion $\Lambda_{v}$-module.
          \item $H^1_\iw(K_\infty,T_v)$ is a $\Lambda_{v}$-module of rank $d$.
    \end{enumerate}
\end{prop}
\begin{proof}It follows from Hypothesis~\ref{cotorsionconjecture} and the Poitou--Tate exact sequence \eqref{eqn:poitoutate} that $\ccX^0_v(A/K_\infty)$ is $\Lambda_v$-torsion.  By \cite[Proposition A.3.2]{Per00}, we have the following exact sequence
 \begin{equation*}\label{eqn:poitoutate2}
  0 \longrightarrow  \ccX^0_v(A/K_\infty)    \longrightarrow  H^2_\iw(K_\infty,T_v) \longrightarrow \prod_{w\in \Sigma} H^2_\iw(K_{\infty,w},T_v).
\end{equation*}
By \cite[Theorem A.2]{lz14}, we see that each $H^2_\iw(K_{\infty,w},T_v)$ is $\Lambda_{v}$-torsion, whence $H^2_\iw(K_\infty,T_v)$ is a torsion $\Lambda_{v}$-module. It can be proven similarly to \cite[Proposition 1.3.2. (i) $\implies$ (ii)]{Per00} that the assertion (i) implies (ii).
\end{proof}

\section{Definition and basic properties of the Kobayashi rank}\label{kobayashiranks}

 Let $E/\bbQ_p$ be a  finite extension with ring of integers $\ccO$. Following \cite[\S 10]{kob03} and \cite[\S 4]{LLZ17}, we define the Kobayashi ranks as follows.
 
\begin{definition}
    Let $(M_n)_{n\geq 1}$ be a projective system of finitely generated $\bbZ_p$-modules. Given an integer $n\ge1$, if $\pi_n:M_n\longrightarrow M_{n-1}$ has finite kernel and cokernel, we  define
\begin{equation*}
  \nabla M_n:=\length_{\ccO}(\ker\pi_n)-\length_{\ccO}(\coker\; \pi_n)+\rank_{\ccO}M_{n-1}.
\end{equation*}
\end{definition}

\begin{lemma}\label{kobayashirankzero}Let $(M'_n)_{n\ge1}$, $(M_n)_{n\ge1}$ and $ (M''_n)_{n\ge1}$ be projective systems of finitely generated $\ccO$-modules.
\begin{enumerate}[label=(\roman*)]
    \item Suppose that for all $n\ge1$, there is an exact sequence
    \begin{equation*}
        0\longrightarrow (M'_n)\longrightarrow (M_n)\longrightarrow (M''_n)\longrightarrow 0.
    \end{equation*}
    If two of $\nabla M_n, \nabla M'_n,\nabla M''_n$ are defined, then the other is also defined, in which case
    \begin{equation*}
        \nabla M_n=\nabla M'_n+\nabla M''_n.
    \end{equation*}
    \item  Suppose that $M_n$ are constant. If $M_n$ is finite or the transition map $M_n\longrightarrow M_{n-1}$ is given by the multiplication map by $p$, then $\nabla M_n=0$.
\end{enumerate}    
\end{lemma}
\begin{proof}
    See \cite[Lemma 10.4]{kob03}.
\end{proof}

\vspace{1.3mm}

For $n\ge1$, let $\zeta_{p^n}$ be a primitive $p^n$-th root of unity and set $\epsilon_n=\zeta_{p^n}-1$. Recall that $\omega_n=(1+X)^{p^n}-1$.

\vspace{1.3mm}

\begin{lemma}\label{kobayashirankandcharacteristicpolynomial}
  Let $F\in \ccO[[X]]$ be a non-zero element. Let $(N_n)_{n\ge1}$ be a projective system given by $N_n= \ccO[[X]]\big/ (F,\omega_{n})$, where the connecting maps are natural projections.

\begin{enumerate}[label=(\roman*)]
    \item Suppose that $F(\epsilon_n)\neq 0$. Then $\nabla N_n$ is defined and is equal to $e\times\ord_{\epsilon_n}F(\epsilon_n)$.
   
     \item Let $M$ be a finitely generated torsion $\ccO[[X]]$-module with the characteristic polynomial $F$, and let $M_n=M/\omega_nM$. Consider the natural projective system $(M_n)_{n\geq 1}$. Then, for $n\gg 0$, $\nabla M_n$ is defined and
    \begin{equation*}
        \nabla M_{n}= e\times\ord_{\epsilon_n}F(\epsilon_n)=e\lambda(M)+(p^n-p^{n-1})\mu(M),
    \end{equation*}
where $e$ is the ramification index of $E/\Qp$; and $\lambda(M)$ and $\mu(M)$ are the Iwasawa invariants of $M$.
\end{enumerate}

\end{lemma}

\begin{proof}
      See \cite[Lemma 4.2]{LLZ17}\footnote{There is a small typo in the statement of \cite[Lemma 4.2 (i)]{LLZ17}.} and \cite[Lemma~10.5]{kob03}.
\end{proof}

Write $p^f$ for the size of the residue field of $E$. If $N$ is a finite $\ccO$-module, then $|N| = p^{f\cdot\length_{\ccO}(N)}$.  Using the following lemma we can relate the growth in the size of a tower of finite $\ccO$-modules and Kobayashi ranks.

\begin{lemma}
    Suppose that $M=(M_n)_{n\geq1}$ is a projective system of finite $\ccO$-modules such that $|M_n|=p^{r_n}$ for some integers $r_n\in f\bbZ$ for all $n\geq 1$. Then $f \nabla M_n = r_n-r_{n-1}$.
\end{lemma}
\begin{proof}
    This is \cite[Lemma 4.3]{LLZ17}.
\end{proof}

\section{Proof of Theorem \ref{maintheorem}}\label{proofofthemaintheorem}

The proof of Theorem~\ref{maintheorem} is divided into a number of steps. Following \cite[\S 10]{kob03}, we define for each integer $n\ge0$
\begin{align*}
    \ccY_v(A/K_n)&:=\coker \left( H^1(G_\Sigma(K_n),T_v) \longrightarrow  \bigoplus_{i=1}^d \dfrac{H^1(K_{n,\frap_i}, T_v)}{A(K_{n,\frap_i})\otimes_{\ccO_F}\ccO_v}\right).
\end{align*}

\subsection{Calculating Kobayashi ranks of certain auxiliary modules} We introduce auxiliary modules in preparation for the calculation of  $\nabla\ccY_v(A/K_n)$. For $\bfs=(s_i)\in \{\sharp,\flat \}^d$, the composition
\begin{equation}\label{lambdahomomorphism}
  \col_\bfs :  H^1_{\iw}(K_\infty,T_v) \xrightarrow{\; (\loc_{\frap_i})_{1\leq i\leq d}\;}  \bigoplus_{i=1}^d H^1_\iw(K_{\infty,\frap_i},T_v) \xrightarrow{\; \left(\col_{v,i}^{s_i}\right)_{1\leq i\leq d}\;}  \Lambda_{v} ^{\oplus d}   
\end{equation}  
 is a $\Lambda_v$-homomorphism between two $\Lambda_v$-modules of rank $d$ (see Proposition \ref{consequenceofcotorsionness}). 
 

\begin{lemma}\label{lem:localizationisinjective}
    The localization map
    \begin{equation*}
        \loc_p:=\left(\loc_{\frap_1},\ldots,\loc_{\frap_d} \right) :   H^1_{\iw}(K_\infty,T_v) \longrightarrow  \bigoplus_{i=1}^d H^1_\iw(K_{\infty,\frap_i},T_v)
    \end{equation*}
    is injective.
\end{lemma}

\begin{proof}
    For $\bfs=(s_i)\in \{\sharp,\flat \}^d$, \eqref{eqn:poitoutate} gives the following exact sequence
     \begin{equation}\label{eqn:exactsequencepoitoutate}
     H^1_{\iw}(K_\infty,T_v) \longrightarrow  \bigoplus_{i=1}^d \dfrac{H^1_\iw(K_{\infty,\frap_i}, T_v)}{H^1_{s_i}(K_{\infty,\frap_i}, T_v)}   \longrightarrow\ccX^\bfs_v(A/K_\infty).
\end{equation}

It follows from Theorem \cite[Proposition 3.2.1]{per94} and Theorem \ref{thm:signedcolemanmaps} that the middle term of the exact sequence \eqref{eqn:exactsequencepoitoutate} is of $\Lambda_v$-rank $d$. Since the $\Lambda_v$-module $H^1_{\iw}(K_\infty,T_v)$ is of rank $d$ (Proposition \ref{consequenceofcotorsionness}), the kernel of the first map in \eqref{eqn:exactsequencepoitoutate} is $\Lambda_v$-torsion as $\ccX^\bfs_v(A/K_\infty)$ is $\Lambda_v$-torsion by  Hypothesis \ref{cotorsionconjecture}. On the other hand, a similar argument as in \cite[Lemma 2.6]{LL22} shows that the $\Lambda_v$-module $H^1_{\iw}(K_\infty,T_v)$ is torsion-free, which implies that the kernel of the first map in \eqref{eqn:exactsequencepoitoutate} is trivial. Consequently, the kernel of $\loc_p$ is trivial.
\end{proof}

\begin{prop}\label{prop:choosingabasisforadiagonalmatix}
    There exist a family of elements $c_1,\ldots,c_d \in H^1_{\iw}(K_\infty,T_v) $ such that 
    \begin{equation*}
        \frac{H^1_{\iw}(K_\infty,T_v)}{\langle c_1,\ldots,c_d \rangle}
    \end{equation*}
  is $\Lambda_v$-torsion,  and $\loc_{\frap_j}(c_i)=0$ whenever $i\neq j$.
\end{prop}
\begin{proof}  Let $R_i$ denote the pre-image of $ H^1_\iw(K_{\infty,\frap_i}, T_v)\subset  \bigoplus_{i=1}^d H^1_\iw(K_{\infty,\frap_i}, T_v)$ under $\loc_p$. It follows from Lemma \ref{lem:localizationisinjective} that $\bigoplus_{i=1}^d R_i$ is isomorphic to a submodule of $H^1_{\iw}(K_\infty,T_v)$. Further, by an abuse of notation, we can identify $\bigoplus_{i=1}^d R_i$ with its image under $\loc_p$ as a submodule of $\bigoplus_{i=1}^d H^1_\iw(K_{\infty,\frap_i}, T_v)$, where $R_i$ is regarded as a $\Lambda_v$-submodule of $H^1_\iw(K_{\infty,\frap_i}, T_v)$ of rank at most $2$.

\vspace{1.3mm}

In the proof of Lemma \ref{lem:localizationisinjective} we see that, for each $i\in\{1,\ldots,d\}$, the cokernel of the composition
     \begin{equation*}
           H^1_{\iw}(K_\infty,T_v) \xrightarrow{\;\;\loc_{\frap_i}\;\;}  H^1_\iw(K_{\infty,\frap_i}, T_v) \xrightarrow{\;\;\col_{v,i}^{s_i}\;\;}\Lambda_v
     \end{equation*}
 is $\Lambda_v$-torsion. It then follows that the $\Lambda_v$-rank of $R_i$ cannot be zero. Therefore, the $\Lambda_v$-module  $\bigoplus_{i=1}^d R_i$ has rank at least $d$. Since  $H^1_{\iw}(K_\infty,T_v)$ is of $\Lambda_v$-rank $d$, we deduce that the $\Lambda_v$-rank of $\bigoplus_{i=1}^d R_i$ is exactly $d$. Thus, the $\Lambda_v$-rank of each $R_i$ is one, and we may pick $c_i\in R_i$ so that $  \displaystyle\frac{H^1_{\iw}(K_\infty,T_v)}{\langle c_1,\ldots,c_d \rangle}$ is $\Lambda_v$-torsion. Further, since $R_i\cap R_j=0$ whenever $i\neq j$, we see that $\loc_{\frap_j}(c_i)=0$ whenever $i\neq j$.
\end{proof}

  Let $H_n^\sharp$ and $H_n^\flat$ denote the entries of the first row of the matrix $\ccH_n$. We then see that $\ccH_n$ is of the form
\begin{equation*}
    \ccH_n=\begin{bmatrix}
        H_n^\sharp& H_n^\flat\\ -\Phi_n H_{n-1}^\sharp & -\Phi_n H_{n-1}^\flat
    \end{bmatrix}.
\end{equation*}
Since $\det\ccH_n=\omega_n/X$, we have
\begin{equation}\label{eq:det}
    -H_n^\sharp H_{n-1}^\flat+H_n^\flat H_{n-1}^\sharp=\omega_{n-1}/X.
\end{equation}

Recall that $\zeta_{p^n}$ denotes a primitive $p^n$-th root of unity and  $\epsilon_n=\zeta_{p^n}-1$.
By  \cite[Proposition 4.6]{LLZ17}, we have for $n\ge1$

  \[\ord_{\epsilon_n}\left( H_n^\sharp(\epsilon_n)\right)=\begin{cases}
     \phi(p^n)\left( r_p + \displaystyle{\sum_{k=1}^{\tfrac{n-1}{2}}}\frac{1}{p^{2k}} \right)  & \text{ if } n \text{ is odd,}\\
     \phi(p^n)\left( \displaystyle{\sum_{k=1}^{\tfrac{n}{2}}}\frac{1}{p^{2k-1}}\right)   & \text{ if } n \text{ is even,}
    \end{cases}\]
    and 
      \[\ord_{\epsilon_n}\left( H_n^\flat(\epsilon_n)\right)=\begin{cases}
     \phi(p^n)\left(  \displaystyle{\sum_{k=1}^{\tfrac{n-1}{2}}}\frac{1}{p^{2k-1}} \right)  & \text{ if } n \text{ is odd,}\\
     \phi(p^n)\left( r_p + \displaystyle{\sum_{k=1}^{\tfrac{n}{2}-1}}\frac{1}{p^{2k}} \right)   & \text{ if } n \text{ is even,}
    \end{cases}\]
    where $r_p=\ord_p(a_p)$, $\phi(p^n)=p^n-p^{n-1}$.

\begin{definition}
     For $n\geq 1$, $i\in\{1,\dots, d\}$, and $u\in\ccO_v^\times$, we define
      \begin{equation*}
         \col_{v,n,i}^{u}:\HIw(K_{\infty,\frap_i}, T_v) \longrightarrow \Lambda_{v,n}
    \end{equation*}
    to be 
    \[
    \begin{bmatrix}
        1&u
    \end{bmatrix}\ccH_n\begin{bmatrix}
        \col^\sharp_{v,i}\\ \col^\flat_{v,i}
    \end{bmatrix}\mod\omega_n.
    \]
\end{definition}

In what follows, we write $\col_{v,n,i}$ for the map given by \eqref{eq:colvn}, after taking $k_\infty$ to be $K_{\infty,\frap_i}$.
\begin{lemma}\label{lem:injects}
   For each $n\ge1$ and $i\in\{1,\dots,d\}$, the map  
   \begin{align*}
\underline{\col}_{v,n,i}:   \HIw(K_{\infty,\frap_i},T_v)&\to \prod_{m=0}^n\cO_v[X]/(\Phi_m)\\
   z&\mapsto\left(\col_{v,n,i}(z)\right)_{0\le m\le n}
   \end{align*}
   induces an injection on $H^1_{/\bff}(K_{n,\frap_i},T_v):=\dfrac{H^1(K_{n,\frap_i}, T_v)}{A(K_{n,\frap_i})\otimes_{\ccO_F}\ccO_v}$.
\end{lemma}
\begin{proof}
As $A(k_\infty)[p^\infty]=0$ (see proof of Lemma~\ref{injectivityoflogarithm}), the projection $\HIw(K_{\infty,\frap_i}, T_v)\to H^1(K_{n,\frap_i},T_v)$ is surjective, resulting in the isomorphism
\[
\HIw(K_{\infty,\frap_i},T_v)/{(\omega_n)}\cong H^1(K_{n,\frap_i},T_v).
\]
In particular, $\underline{\col}_{v,n,i}$ induces a map on $H^1(K_{n,\frap_i},T_v)$.

Let $\theta$ be a character of $G_n=\gal(K_{n,\frap_i}/K_{\frap_i})$. We write $e_\theta\in F_v[G_n]$ for the idempotent element of $\theta$. Let $0\le m\le n$ be the integer such that $\theta$ factors through $G_m$ but not $G_{m-1}$. Then, $e_\theta\cdot z $ belongs to the image of $e_\theta\cdot A(K_{n,\frap_i})\otimes_{\ccO_F}\ccO_v$ if and only if $\col_{v,m,i}(z)\equiv 0\mod \Phi_m$ since $\col_{v,m,i}$ is defined by pairings with $d_n$, and $d_n$ generates $\widehat{A}(K_{m,\frap_i})/\widehat{A}(K_{m-1,\frap_i})$ as an $\cO_v[G_m]$-module (see the proof of Proposition~\ref{pointsgeneratetheformalgroup}). Therefore, $z\in\ker\underline{\col}_{v,n,i}$ if and only if $e_\theta\cdot z\in e_\theta\cdot A(K_{n,\frap_i})\otimes_{\ccO_F}\ccO_v$ for all $\theta$. This tells us that the kernel of $\underline{\col}_{v,n,i}$ is the image of $ A(K_{n,\frap_i})\otimes_{\ccO_F}\ccO_v$ in $H^1(K_{n,\frap_i},T_v)$, as desired.
\end{proof}

\begin{lemma}\label{lem:choosingaunit}
    For $n\geq 1$ and $i\in\{1,\dots,d\}$, we have
    \begin{enumerate}[label=(\roman*)]
        \item $\omega_{n-1}\Lambda_v\subset  \image(\col_{v,n,i}^{u})$ for any choice of $u$;
        \item There exist infinitely many $u \in \ccO_v^\times$ such that $ \col_{v,n,i}^{u}$ induces an injection
        \begin{equation*}
              H^1_{/\bff}(K_{n,\frap_i}, T_v) \xhookrightarrow{\;\;\;\;\;\;\;\;} \Lambda_{v,n}
        \end{equation*}
        with finite cokernel.
    \end{enumerate}
\end{lemma}
\begin{proof}
     By Theorem~\ref{thm:imageofsumofcolemanmaps}, there exists $z\in\HIw(K_{\infty,\frap_i},T_v)$ such that   $\col_{v,i}^\sharp(z)=-XH_{n-1}^\flat$ and  $\col_{v,i}^\flat(z)=XH_{n-1}^\sharp$. For such $z$, we have
    \begin{align*}
    \col_{v,n,i}^u(z)&=\begin{bmatrix}
        1& u
    \end{bmatrix}\begin{bmatrix}
        H_{n}^\sharp & H_{n}^\flat\\-\Phi_nH_{n-1}^\sharp & -\Phi_nH_{n-1}^\flat
    \end{bmatrix}
    \begin{bmatrix}
        -XH_{n-1}^\flat\\XH_{n-1}^\sharp
    \end{bmatrix}\\
        &=-X(H_n^\sharp H_{n-1}^\flat - H_{n}^\flat H_{n-1}^\sharp)\\
        &=\omega_{n-1}
    \end{align*}
    by \eqref{eq:det}. Therefore, assertion (i) holds.

We now prove assertion (ii).    Recall from Theorem~\ref{thm:signedcolemanmaps} that
     \[ \ccH_n \begin{bmatrix}
        \col^{\sharp}_{v,i}\\ \col^{\flat}_{v,i}
    \end{bmatrix} \equiv \begin{bmatrix}
        \col_{v,n,i}\\ -\xi_{n-1}\col_{v,n-1,i}
    \end{bmatrix} \mod \omega_n,\]
    which implies 
    \[
    \col_{v,n,i}^u=\col_{v,n,i}-u\xi_{n-1}\col_{v,n-1,i}.
    \]

By the Chinese remainder theorem, there is an injection $\displaystyle\Lambda_{v,n}\hookrightarrow \prod_{m=0}^n\cO[X]/(\Phi_m)$. For all $1\le m\le n$, we have
\[
\ccC_r(X)\equiv\begin{bmatrix}
    a_p&1\\-p&0
\end{bmatrix}\mod\omega_m,\quad m<r\le n.
\]
Thus,
\[
\col_{v,n,i}^u\equiv\begin{bmatrix}
    1&u
\end{bmatrix} \begin{bmatrix}
    a_p&1\\-p&0
\end{bmatrix}^{n-m}\ccH_m  \begin{bmatrix}
        \col^{\sharp}_{v,i}\\ \col^{\flat}_{v,i}
    \end{bmatrix}   \mod \omega_m,
\]
which in turn implies that
\[
\col_{v,n,i}^u\equiv\begin{bmatrix}
    1&u
\end{bmatrix} \begin{bmatrix}
    a_p&1\\-p&0
\end{bmatrix}^{n-m} \begin{bmatrix}
        \col_{v,m,i}\\ 0
    \end{bmatrix}   \mod \Phi_m.
\]
Therefore, there are infinitely many choices of $u$ such that $\col_{v,n,i}^u$ is congruent to a non-zero multiple of $\col_{v,m,i}$ modulo $\Phi_m$. A similar argument shows that the same is true for $m=0$.
In particular, there exist infinitely many $u$ such that this holds for all $0\le m\le n$. Once such $u$ is picked, Lemma~\ref{lem:injects} tells us that $\col_{v,n,i}^u$ induces an injection on $H^1_{/\bff}(K_{n,\frap_i},T_v)$.

Finally, as  $ H^1_{/\bff}(K_{n,\frap_i}, T_v)$ and $\Lambda_{v,n}$ are both of rank $p^n$ over $\cO_v$, the cokernel is finite.
\end{proof}

\begin{corollary}\label{cor:kobayshirankofcolemanwithunit}
   For $i\in\{1,\dots,d\}$,  let $z_i\in \HIw(K_{\infty,\frap_i},T_v)$ and $u\in \ccO_v^\times$ satisfying Lemma \ref{lem:choosingaunit} (ii). For $n\geq 1$, write $M_{n,i}$ for the $\Lambda_{v,n}$-module generated by the natural image of $z$ in $H^1_{/\bff}(K_{n,\frap_i}, T_v)$. Suppose that the natural projection
   \begin{equation*}
        \pi_i: \frac{H^1_{/\bff}(K_{n,\frap_i}, T_v)}{M_{i,n}} \longrightarrow \frac{H^1_{/\bff}(K_{n-1,\frap_i}, T_v)}{M_{i,n-1}}
   \end{equation*}
   has finite kernel for some $n$. Then we have
   \begin{equation*}
       \nabla\left(\frac{H^1_{/\bff}(K_{n,\frap_i}, T_v)}{M_{i,n}} \right)=\ord_{\epsilon_n}\left(\col_{v,n,i}^{u}(z)(\epsilon_n)\right).
   \end{equation*}
\end{corollary}
\begin{proof}
The proof is identical to the proof of \cite[Corollary 4.5]{LL22}.
 \end{proof}

\begin{prop}\label{prop:orderofunitcolemanmap}
  For $n\geq 1$, let $z\in \HIw(K_{\infty,\frap_i},T_v)$ such that $\col_{v,n,i}^{u}(z)(\epsilon_n)\neq 0$ and $u\in \ccO_v^\times$ satisfying Lemma \ref{lem:choosingaunit} (ii).  Then we have
    \begin{equation*}
     \ord_{\epsilon_n}\left(\col_{v,n,i}^{u}(z)(\epsilon_n)\right) = \ord_{\epsilon_n}\left( H_n^{s_i}(\epsilon_n)\right)+ \ord_{\epsilon_n}\left( \col_{v,i}^{s_i}(z)(\epsilon_n)\right),
    \end{equation*}
    where $s_i\in\{\sharp,\flat\}$ depends on the parity of $n$.    
\end{prop}
\begin{proof}
      See \cite[Proposition 4.6]{LL22}, which is a special case of \cite[Proposition 4.6 and Corollary 4.8]{LLZ17}.
\end{proof}

 From now on, we fix $\bfc=(c_1,\ldots ,c_{d})$ satisfying the properties given by Proposition \ref{prop:choosingabasisforadiagonalmatix}.

\begin{definition}
    Write $M_\bfc$ for the $\Lambda_v$-module generated by $ c_1,\ldots ,c_{d}$.  We define for each integer $n\ge0$ the following $\ccO_v$-modules
\begin{align*}
    \ccY_v'(A/K_n)&:=\coker \left( H^1_\iw(K_\infty,T_v)_{\Gamma_n} \longrightarrow  \bigoplus_{i=1}^d H^1_{/\bff}(K_{n,\frap_i},T_v)\right),\\
     \ccY_v''(A/K_n)&:=\coker \left(\left(M_\bfc \right)_{\Gamma_n} \longrightarrow  \bigoplus_{i=1}^dH^1_{/\bff}(K_{n,\frap_i},T_v)\right).
\end{align*}
\end{definition}

\begin{prop}\label{prop:kobayashirankofY''}
    For $n\gg 0$, $\nabla\ccY''_v(A/K_n)$ is defined. Moreover, there exists $\bfs=(s_i)_{1\leq i\leq d}\in \{\sharp,\flat\}^d$ that depends only on the parity of $n$ such that, for $n\gg 0$, we have
    \[  \nabla\ccY''_v(A/K_n)=     F_v(\bfs,n)+ 
    \nabla\left(\frac{\Lambda_v^{\oplus d}}{\col_\bfs\left(M_\bfc\right)} \right)_{\Gamma_n}
    , \]
    where
    \[
         F_v(\bfs,n)= 
              \begin{cases}
    \phi(p^n)\cdot\left( a_\sharp \cdot \left( r_p + \displaystyle{\sum_{k=1}^{\tfrac{n-1}{2}}}\frac{1}{p^{2k}} \right) + a_\flat \cdot \left(  \displaystyle{\sum_{k=1}^{\tfrac{n-1}{2}}}\frac{1}{p^{2k-1}}\right) \right)   & \text{ if } n \text{ is odd,}\\
   \phi(p^n)\cdot\left( a_\sharp \cdot \left(  \displaystyle{\sum_{k=1}^{\tfrac{n}{2}}}\frac{1}{p^{2k-1}} \right) + a_\flat \cdot \left(  r_p + \displaystyle{\sum_{k=1}^{\tfrac{n}{2}-1}}\frac{1}{p^{2k}} \right) \right)    & \text{ if } n \text{ is even,}
    \end{cases}
         \]
        where $a_\sharp$ (resp. $a_\flat$) is the number of $s_i=\sharp$ (resp. $s_i=\flat$).
\end{prop}
\begin{proof}
The proof is similar to the proof of \cite[Proposition 5.3]{LL22}. By definition, we have
\begin{equation}\label{eqn:directsumforY''}
     \ccY''_v(A/K_n)= \bigoplus_{i=1}^d\frac{H^1_{/\bff}(K_{n,\frap_i},T_v)}{(M_i)_{\Gamma_n}},
\end{equation}
where $M_i$ denote the image of $M_\bfc$ in $\HIw(K_{\infty,\frap_i},T_v)$. Let $\pi_i$ denote the natural projection
   \begin{equation*}
        \pi_i: \frac{H^1_{/\bff}(K_{n,\frap_i}, T_v)}{(M_{i})_{\Gamma_n}} \longrightarrow \frac{H^1_{/\bff}(K_{n-1,\frap_i}, T_v)}{(M_{i})_{\Gamma_{n-1}}}.
   \end{equation*}
   We deduce from Corollary \ref{cor:kobayshirankofcolemanwithunit} that $\ker \pi_i$ is finite if and only if $\ord_{\epsilon_n}\left(\col_{v,n,i}^{u}(c_i)(\epsilon_n)\right)$ does not vanish at $\epsilon_n$.

   Let  $\bfs=(s_i)_{1\leq i\leq d}\in \{\sharp,\flat\}^d$ be given as in Proposition \ref{prop:orderofunitcolemanmap}, which depends only on the parity of $n$. It follows from the exact sequence \eqref{eqn:exactsequencepoitoutate} and Hypothesis \ref{cotorsionconjecture} that  $\col_{v,n,i}^{u}(c_i)$ is a non-zero element of $\Lambda_{v,n}$. In particular, for $n\gg 0$, we have  $\col_{v,n,i}^{u}(c_i)(\epsilon_n)\neq 0$. It then follows that $\ker \pi_i$ is finite and $\displaystyle\nabla\left( \frac{H^1_{/\bff}(K_{n,\frap_i},T_v)}{(M_i)_{\Gamma_n}} \right)$ is defined. It follows from Proposition \ref{prop:orderofunitcolemanmap} and Corollary \ref{cor:kobayshirankofcolemanwithunit} that
   \begin{equation*}
       \nabla\left( \frac{H^1_{/\bff}(K_{n,\frap_i},T_v)}{(M_i)_{\Gamma_n}} \right)=\ord_{\epsilon_n}\left( H_n^{s_i}(\epsilon_n)\right)+ \ord_{\epsilon_n}\left( \col_{v,i}^{s_i}(c_i)(\epsilon_n)\right).
   \end{equation*}
   The assertion of the proposition now follows from Lemma \ref{kobayashirankzero} applied to the direct sum \eqref{eqn:directsumforY''}.
\end{proof}

\begin{remark}
    If $p$ is an odd prime and $E/\bbQ$ is a $p$-supersigular elliptic curve with $a_p(E) = 0$, then we have
 \[
         F_v(\bfs,n)= 
             \begin{cases}
    \phi(p^n)\cdot\left(  \displaystyle{\sum_{k=1}^{\tfrac{n-1}{2}}}\frac{1}{p^{2k-1}}\right)    & \text{ if } n \text{ is odd,}\\
   \phi(p^n) \cdot \left(  \displaystyle{\sum_{k=1}^{\tfrac{n}{2}}}\frac{1}{p^{2k-1}} \right)     & \text{ if } n \text{ is even.}
    \end{cases}
         \]
\end{remark}

\begin{corollary}\label{cor:kobayashirankofY}
 For $n\gg0$, $\nabla \ccY_v(A/K_n)$ is defined and is equal to
      \[  \nabla\ccY_v(A/K_n)=   F_v(\bfs,n)+\nabla \left(\coker \left(\col_\bfs\right)\right)_{\Gamma_n}, \]
    where $\bfs=(s_i)_{1\leq i\leq d}\in \{\sharp,\flat\}^d$ and $F_v(\bfs,n)$ are given in Proposition \ref{prop:kobayashirankofY''}.
\end{corollary}
\begin{proof}
         Consider the following short exact sequences
         \begin{align*}
    0&\longrightarrow \frac{H^1_\iw(K_\infty,T_v)_{\Gamma_n}}{\left(M_\bfc\right)_{\Gamma_n}}\longrightarrow  \left(\frac{\Lambda_v^{\oplus d}}{\col_\bfs\left(M_\bfc\right)} \right)_{\Gamma_n}   \longrightarrow  \left(\coker \left(\col_\bfs\right)\right)_{\Gamma_n}  \longrightarrow 0,\\
    0&\longrightarrow \frac{H^1_\iw(K_\infty,T_v)_{\Gamma_n}}{\left(M_\bfc\right)_{\Gamma_n}}\longrightarrow \ccY''_v(A/K_n)\longrightarrow\ccY'_v(A/K_n)\longrightarrow 0.
\end{align*}
 Furthermore, $\displaystyle\frac{H^1_\iw(K_\infty,T_v)_{\Gamma_n}}{\left(M_\bfc\right)_{\Gamma_n}}$ is isomorphic to $\displaystyle\left(\frac{H^1_\iw(K_\infty,T_v)}{M_\bfc}\right)_{\Gamma_n}$. Since $\displaystyle\frac{H^1_\iw(K_\infty,T_v)}{M_\bfc}$ is $\Lambda_v$-torsion, we see that $\displaystyle\nabla  \frac{H^1_\iw(K_\infty,T_v)_{\Gamma_n}}{\left(M_\bfc\right)_{\Gamma_n}}$ is defined for $n\gg 0$. It then follows from Proposition \ref{prop:kobayashirankofY''} and Lemma \ref{kobayashirankzero} that $\nabla\ccY_v'(A/K_n)$ is defined and is equal to

 \begin{equation*}
     \nabla \frac{H^1_\iw(K_\infty,T_v)_{\Gamma_n}}{\left(M_\bfc\right)_{\Gamma_n}}= \nabla \ccY''_v(A/K_n)-\nabla\ccY'_v(A/K_n).
 \end{equation*}

 It follows from Proposition \ref{prop:kobayashirankofY''} that $\displaystyle\nabla \left(\frac{\Lambda_v^{\oplus d}}{\col_\bfs\left(M_\bfc\right)} \right)_{\Gamma_n} $ is defined for $n\gg 0$. Thus, Proposition \ref{kobayashirankzero} tells us that
 \begin{equation*}
  \nabla\left(\coker \left(\col_\bfs\right)\right)_{\Gamma_n}= \nabla \left(\frac{\Lambda_v^{\oplus d}}{\col_\bfs\left(M_\bfc\right)} \right)_{\Gamma_n} -\nabla   \frac{H^1_\iw(K_\infty,T_v)_{\Gamma_n}}{\left(M_\bfc\right)_{\Gamma_n}} .
 \end{equation*}

By an adaptation of \cite[Proposition 10.6 i)]{kob03} and \cite[Proposition 3.6]{LL22} to our setting that, we have for $n\gg0$
\begin{equation*}
    \nabla \ccY_v(A/K_n)= \nabla \ccY'_v(A/K_n).
\end{equation*}
 The assertion of the corollary now follows from combining the last three equalities.
\end{proof}

\subsection{Final steps of the proof of Theorem \ref{maintheorem}}

We are now ready to conclude the proof of Theorem \ref{maintheorem}. Recall that $\ccX_v(A/K_\infty)$ and $\ccX^0_v(A/K_\infty)$ denote the Pontryagin duals of $\Sel_{v}(A/K_\infty)$ and  $ \Sel^0_v(A/K_\infty)$, respectively. 

\vspace{1.3mm}

\begin{lemma}\label{controltheorem}
 The natural restriction map
    \begin{equation*}
         \Sel^0_v(A/K_n) \longrightarrow  \Sel^0_v(A/K_\infty)^{\Gamma_n}
    \end{equation*}
    is injective and has finite cokernel. Furthermore, the cardinality of the cokernel stabilizes as $n\rightarrow\infty$.
\end{lemma}
\begin{proof}
The assertion follows from \cite[Theorem 3.3]{Lim20}.
\end{proof}

\begin{lemma}\label{kobayashirankofX0}
Assume that Hypothesis~\ref{cotorsionconjecture} holds.   Let $n\geq 0$ be an integer.
    \begin{enumerate}[label=(\roman*)]
        \item We have the following short exact sequence:
        \begin{equation*}
            0 \longrightarrow \ccY_v(A/K_n) \longrightarrow \ccX_v(A/K_n) \longrightarrow \ccX^0_v(A/K_n) \longrightarrow 0.
        \end{equation*}

\vspace{1.5mm}
        
        \item  For $n\gg 0$, $\nabla \ccX^0_v(A/K_n)$ is defined and satisfies the equality
        \begin{equation*}
            \nabla \ccX^0_v(A/K_n)= \nabla \ccX^0_v(A/K_\infty)_{\Gamma_n}.
        \end{equation*}
    \end{enumerate}
\end{lemma}
\begin{proof}
   Assertion (i) follows from the exact sequence \eqref{eqn:poitoutateovern} and the definition of $\ccY_v(A/K_n)$. We saw in the proof of Proposition \ref{consequenceofcotorsionness} that $\ccX^0_v(A/K_\infty)$ is $\Lambda_{v}$-torsion. By Lemma~\ref{controltheorem}, the kernel and cokernel of the natural map
    \begin{equation*}
        \ccX^0_v(A/K_\infty)_{\Gamma_n} \longrightarrow \ccX^0_v(A/K_n)
    \end{equation*}
    are finite with cardinality independent of $n$ when $n$ is sufficiently large. Combining this with Lemma \ref{kobayashirankandcharacteristicpolynomial}(ii), the second assertion follows.
\end{proof}

Let $e_v$ denote the ramification index of $F_v/\Qp$.

 \begin{prop}\label{prop:nablaX}
 Suppose that Hypothesis~\ref{cotorsionconjecture} holds. For $n\gg 0$, we have
    \begin{equation*}
       \nabla \ccX_v(A/K_n)=  F_v(\bfs,n)+e_v\lambda_{v,\bfs} + (p^n-p^{n-1})\mu_{v,\bfs},
    \end{equation*}
     where $\bfs=(s_i)_{1\leq i\leq d}\in \{\sharp,\flat\}^d$ and $F_v(\bfs,n)$ are given in Proposition \ref{prop:kobayashirankofY''}, which depend on the parity of $n$.
    \end{prop}
         \begin{proof}
 Let $\bfs=(s_i)_{1\leq i\leq d}\in \{\sharp,\flat\}^d$ be given as in Proposition \ref{prop:kobayashirankofY''}. We have the following Poitou--tate exact sequence of $\Lambda_v$-torsion modules:
   \[
 0\longrightarrow \coker\left(\col_\bfs\right) \longrightarrow \ccX_v^\bfs(A/K_\infty) \longrightarrow \ccX_v^0(A/K_\infty) \longrightarrow0.
 \]

 By \cite[Lemma 10.3]{kob03}, we obtain the following six-term exact sequence
  \begin{align*}
 0\longrightarrow  \coker&\left(\col_\bfs\right)^{\Gamma_n} \longrightarrow \ccX_v^\bfs(A/K_\infty)^{\Gamma_n} \longrightarrow \ccX_v^0(A/K_\infty)^{\Gamma_n} \\ 
&\longrightarrow \left(\coker\left(\col_\bfs\right)\right)_{\Gamma_n} \longrightarrow \ccX_v^\bfs(A/K_\infty)_{\Gamma_n} \longrightarrow \ccX_v^0(A/K_\infty)_{\Gamma_n} \longrightarrow0.
  \end{align*}

 By \cite[Lemma 3.5]{LL22}, the Kobayashi ranks of the first three terms vanish for $n\gg 0$. Hence, it follows from Lemma \ref{kobayashirankzero} (i) that, for $n\gg0$, we have
\begin{equation*}
    \nabla\ccX_v^\bfs(A/K_\infty)_{\Gamma_n} = \nabla \left(\coker\left(\col_\bfs\right)\right)_{\Gamma_n}+\nabla \ccX_v^0(A/K_\infty)_{\Gamma_n}.
\end{equation*}
Combining this equality with Lemma \ref{kobayashirankofX0} and Corollary \ref{cor:kobayashirankofY}, we deduce that
\begin{equation*}
    \nabla\ccX_v(A/K_\infty)_{\Gamma_n} =F_v(\bfs,n)+\nabla \ccX_v^\bfs(A/K_\infty)_{\Gamma_n}.
\end{equation*}
The assertion of the proposition now follows from Lemma \ref{kobayashirankandcharacteristicpolynomial} (ii) under Hypothesis~\ref{cotorsionconjecture}.
     \end{proof}

\begin{theorem}
       Let $A$ be an abelian variety of $\gl_2$-type over $\bbQ$ associated to a weight-two newform with trivial nebentype. Let $K$ be a number field and $p$ be an odd prime satisfying Hypothesis \ref{splittingtype}. Furthermore, assume that Hypothesis \ref{cotorsionconjecture} holds. Then the following statements hold.
    \begin{itemize}
        \item[(i)] $\rank_{\ccO_F} A(K_n)$ is bounded independently of $n$;
        \item [(ii)] Assume that $\Sha(A/K_n )[v^\infty]$ is finite for all $n\geq 0$. Define \[\displaystyle r_\infty=\sup_{n\ge0}\left\{ \rank_{\ccO_F} A(K_n)\right\}.\] Then, for $n\gg 0$, we have
        \[\nabla\Sha(A/K_n )[v^\infty]=  F_v(\bfs,n) +(p^n-p^{n-1})\mu_{v,\bfs} + e_v\lambda_{v,\bfs} -r_\infty. \]
    \end{itemize}
\end{theorem}
\begin{proof}
    Since $\nabla\ccX_v(A/K_n)$ is defined when $n$ is sufficiently large by Proposition~\ref{prop:nablaX}, the kernel and cokernel of the natural map $\ccX_v(A/K_{n+1})\rightarrow \ccX_v(A/K_n)$ are finite for such $n$. In particular, $\rank_{\ccO_v}\ccX_v(A/K_n)$ is bounded independently of $n$. Hence, the first assertion follows from the well-known exact sequence
    \begin{equation}\label{selmergroupexactsequence}
            0 \longrightarrow A(K_n)\otimes_{\ccO_F} F_v/\ccO_v \longrightarrow \Sel_v(A/K_n) \longrightarrow \Sha(A/K_n )[v^\infty] \longrightarrow 0.
        \end{equation}

         It follows from the exact sequence \eqref{selmergroupexactsequence} and Lemma \ref{kobayashirankzero} that
    \begin{equation*}
        \nabla\Sha_v(A/K_n )[v^\infty] = \nabla \ccX_v(A/K_n)-\nabla A(K_n)\otimes_{\ccO_F} \ccO_v.
    \end{equation*}
    Assertion (i) tells us that $\nabla A(K_n)\otimes_{\ccO_F} \ccO_v=r_\infty$ for $n\gg0$.  Thus, assertion (ii) follows from Proposition~\ref{prop:nablaX}.
\end{proof}

\appendix
\section{Honda-type of formal groups}\label{app}
Let $R$ be a commutative ring. Let $n\ge1$ be an integer, and let $\bfx$ be the column vector of $n$ variables $(X_1,\ldots,X_n)^t$. Define $R[[\bfx]]:=R[[X_1,\dots,X_n]]$ to be the ring of formal power series in $n$ variables. Write $R[[\bfx]]^n$ for the set of column vectors $\left(\varphi_1(\bfx),\ldots,\varphi_n(\bfx) \right)^t$ with $\varphi_i(\bfx)\in R[[\bfx]]$. If $\varphi(\bfx),\psi(\bfx)\in R[[\bfx]]^n$ differ only in terms of total degree $\geq r$, we write $\varphi(\bfx)\equiv \psi(\bfx) \mod \deg r$.  We write $R[[\bfx]]^n_0$ for the set $\{f\in R[[\bfx]]^n \mid f(\bfx) \equiv 0 \mod \deg 1 \}$.  

 Define the identity function $i$ of $R[[\bfx]]^n_0$ by $i(\bfx)=\bfx$. If $\varphi(\bfx)$ is an element of $R[[\bfx]]^n_0$ such that $\varphi(\bfx)\equiv P\bfx \mod \deg 2$ for some invertible matrix $P\in M_{n\times n}(R)$, there is a unique element $\varphi^{-1}(\bfx)\in R[[\bfx]]^n_0$ satisfying $\varphi\circ \varphi^{-1} =\varphi^{-1} \circ \varphi =i$. We say that $\varphi$ is invertible and $\varphi^{-1}$ is called the inverse of $\varphi$.

\vspace{1.3mm}

\begin{definition}
     Let $\bfx=(X_1,\dots,X_n)^t$ and $\bfy=(Y_1,\dots,Y_n)^t$ be vectors of $n$ variables. An $n$-\textbf{dimensional formal group} over $R$ is a formal power series $\hatF(\bfx,\bfy)\in R[[\bfx,\bfy]]^n_0$ satisfying the following conditions
\begin{enumerate}[label=(\roman*)]
    \item $\hatF(\bfx,\bfy)\equiv \bfx+\bfy \mod \deg 2 $,
    \item $\hatF(\hatF(\bfx,\bfy),\bfz)=\hatF(\bfx,\hatF(\bfy,\bfz))$,
    \item $\hatF(\bfx,\bfy)=\hatF(\bfy,\bfx)$.
\end{enumerate}

\end{definition}

\begin{definition}
    Let $\hatF$ and $\hatG$ be formal groups over $R$, of dimension $n$ and $m$, respectively. An element $\varphi\in R[[\bfx]]^m_0$ is said to be a homomorphism of $\hatF$ to $\hatG$, if $\varphi$ satisfies $\varphi\circ \hatF=\hatG\circ \varphi$, where $(\hatG\circ \varphi)(\bfx,\bfy)$ stands for $\hatG(\varphi(\bfx),\varphi(\bfy))$. 

    If $m=n$ and $\varphi$ is invertible, $\varphi$ is called an isomorphism. If there is an isomorphism $\varphi$ of
$\hatF$ to $\hatG$ such that $\varphi(\bfx)\equiv \bfx$ mod $\deg 2$, we shall say that $\hatG$ is \textbf{strongly isomorphic}
to $\hatF$ and write $\varphi$ : $\hatF\approx \hatG$ over $S$.
\end{definition}

\begin{theorem}
    Let $\hatF(\bfx,\bfy)$ be an $n$-dimensional formal group over $R$. Then there exists a unique element $\ell \in R[[\bfx]]^n_0$ satisfying
    \begin{equation*}
        \ell(\bfx)\equiv \bfx \mod \deg 2
    \end{equation*}
    and 
    \begin{equation*}
        \hatF(\bfx,\bfy)=\ell^{-1}\left(\ell(\bfx)+\ell(\bfy) \right).
    \end{equation*}
    In particular, $\hatF(\bfx,\bfy)\approx \bfx+\bfy$ over $R$.
\end{theorem}
\begin{proof}
    This is \cite[Theorem 1]{Hon70}.
\end{proof}

Such an element $\ell$ is called the \textbf{logarithmic function} of the formal group $\hatF$. 

\vspace{1.3mm}

Let $p$ be a fixed prime number. Let $M_{n\times n}\left(\bbZ_p\right)[[T ]]$ be the ring of one-variable formal power series with coefficients in the ring of $n\times n$ matrices with entries in $\bbZ_p$. 

\begin{definition}\label{definitionofu*l}
    For $u=\sum_{i\geq 0}B_i T^i \in M_{n\times n}\left(\bbZ_p\right)[[T ]]$ and $\ell(\bfx) \in \Qp[[\bfx]]^n_0$, we define
\begin{equation*}
    u* \ell(\bfx):=\sum_{i\geq 0} B_i \ell(\bfx^{p^i}),
\end{equation*}
where $\bfx^{p^i}$ denotes $\left(X^{p^i}_1,\dots, X_n^{p^i} \right)^t$.
\end{definition}

\begin{definition}
   For an element $u=\sum_{i\geq 0}B_i T^i \in M_{n\times n}(\bbZ_p)[[T]]$ such that $B_0 = pI_n$, an element $\ell(\bfx) \in \Qp[[\bfx]]^n_0$ is said to be of \textbf{Honda type} $u$ if 
\begin{itemize}
    \item $\ell(\bfx) \equiv \bfx \mod \deg 2$;
    \item  $u* \ell (\bfx)\equiv 0 \mod p$.
\end{itemize}

An element $u$ as above is called a \text{special} element. For a special element $u$, a formal group $\hatF$ over $\bbQ_p$ is said to be of \textbf{Honda type} $u$ if its logarithmic function is of type $u$.
\end{definition}

\begin{theorem}[Honda]\label{hondaisomorphismtheorem}

The following statement holds.
      \begin{enumerate}[label=(\roman*)]
      \item Let $\ell(\bfx) \in \Qp[[\bfx]]^n_0$ be of Honda type $u \in M_{n\times n}(\bbZ_p)[[T]]$. Then $  \hatF(\bfx,\bfy)=\ell^{-1}\left(\ell(\bfx)+\ell(\bfy) \right)$ is a formal group over $\bbZ_p$.
          \item Let $u_1$ and $u_2$ be special elements. Assume that $\hatF$ and $\hatG$ are formal groups of Honda type $u_1$ and $u_2$, respectively. Then there exists an element $v$ in $I_n+TM_{n\times n}(\bbZ_p)[[T]]$ satisfying $u_1=v u_2$ if and only if $\hatF$ and $\hatG$ are strongly isomorphic to  over $\bbZ_p$.
          \end{enumerate}
\end{theorem}
\begin{proof}
    This is \cite[Theorems 2--4]{Hon70}.
\end{proof}
\begin{theorem}\label{hondatypeofLseries}
Let $\{C_q,C_{q^2}:q\text{ prime}\}$ be a set of commuting matrices in $M_{n\times n}(\bbZ)$. Consider the formal Dirichlet series
\begin{equation*}
    \sum_{n\geq 1}\dfrac{C_n}{n^s}:=\prod_q\left( I_n - C_qp^{-s}+C_{q^2}p^{1-2s}\right)^{-1},
\end{equation*}
where $C_n\in M_{n\times n}(\bbZ)$. Put
\begin{equation*}
    \ell (\bfx):= \sum_{n\geq 1}\dfrac{C_n}{n}\bfx^n\in \bbQ[[\bfx]]^n_0.
\end{equation*}

    The formal power series
    \begin{equation*}
        \hatF(\bfx,\bfy):=\ell ^{-1}\left(\ell (\bfx)+ \ell(\bfy)\right)
    \end{equation*}
    is an $n$-dimensional formal group over $\bbZ$. For each prime number $p$, it is of the Honda type $pI_n-C_pT+C_{p^2}T^2$.
\end{theorem}
\begin{proof}
    This is \cite[Theorem 8]{Hon70}.
\end{proof}

\bibliographystyle{alpha} 
\bibliography{reference}

\end{document}